\documentclass[12pt]{article}

\usepackage{amssymb, amsmath, amsthm, amscd}

\usepackage[pdftex]{graphics}

\usepackage[utf8]{inputenc}

\usepackage[all,cmtip]{xy}

\usepackage{enumitem}

\usepackage{setspace}

\usepackage{mathdots}

\usepackage[mathscr]{euscript}

\usepackage[colorlinks=true]{hyperref}

\hypersetup{colorlinks   = true}

\hypersetup{linkcolor=blue}

\usepackage{tikz}

\usetikzlibrary{cd}

\newcommand{\hooklongrightarrow}{\lhook\joinrel\longrightarrow}

\begin{document}

\newtheorem{The}{Theorem}[section]

\newtheorem{Lem}[The]{Lemma}

\newtheorem{Prop}[The]{Proposition}

\newtheorem{Cor}[The]{Corollary}

\newtheorem{Rem}[The]{Remark}

\newtheorem{Obs}[The]{Observation}

\newtheorem{SConj}[The]{Standard Conjecture}

\newtheorem{Titre}[The]{\!\!\!\! }

\newtheorem{Conj}[The]{Conjecture}

\newtheorem{Question}[The]{Question}

\newtheorem{Prob}[The]{Problem}

\newtheorem{Def}[The]{Definition}

\newtheorem{Not}[The]{Notation}

\newtheorem{Claim}[The]{Claim}

\newtheorem{Conc}[The]{Conclusion}

\newtheorem{Ex}[The]{Example}

\newtheorem{Fact}[The]{Fact}

\newtheorem{Formula}[The]{Formula}

\newtheorem{Formulae}[The]{Formulae}

\newtheorem{The-Def}[The]{Theorem and Definition}

\newtheorem{Prop-Def}[The]{Proposition and Definition}

\newtheorem{Cor-Def}[The]{Corollary and Definition}

\newtheorem{Conc-Def}[The]{Conclusion and Definition}

\newtheorem{Terminology}[The]{Note on terminology}

\newtheorem{Construction}[The]{Construction}

\newcommand{\C}{\mathbb{C}}

\newcommand{\R}{\mathbb{R}}

\newcommand{\N}{\mathbb{N}}

\newcommand{\Z}{\mathbb{Z}}

\newcommand{\Q}{\mathbb{Q}}

\newcommand{\Proj}{\mathbb{P}}

\newcommand{\Rc}{\mathcal{R}}

\newcommand{\Oc}{\mathcal{O}}

\newcommand{\Vc}{\mathcal{V}}

\newcommand{\Id}{\operatorname{Id}}

\newcommand{\pr}{\operatorname{pr}}

\newcommand{\rk}{\operatorname{rk}}

\newcommand{\del}{\partial}

\newcommand{\delbar}{\bar{\partial}}

\newcommand{\Cdot}{{\raisebox{-0.7ex}[0pt][0pt]{\scalebox{2.0}{$\cdot$}}}}

\newcommand\nilm{\Gamma\backslash G}

\newcommand\frg{{\mathfrak g}}

\newcommand{\fg}{\mathfrak g}

\newcommand{\Oh}{\mathcal{O}}

\newcommand{\Kur}{\operatorname{Kur}}

\newcommand\gc{\frg_\mathbb{C}}

\newcommand\jonas[1]{{\textcolor{green}{#1}}}

\newcommand\luis[1]{{\textcolor{red}{#1}}}

\newcommand\dan[1]{{\textcolor{blue}{#1}}}

\begin{center}

{\Large\bf Deformations of Higher-Page Analogues of $\partial\bar\partial$-Manifolds}

\end{center}

\begin{center}

{\large Dan Popovici, Jonas Stelzig and Luis Ugarte}

\end{center}

\vspace{1ex}

\noindent{\small{\bf Abstract.} We extend the notion of small essential deformations of Calabi-Yau complex structures from the case of the Iwasawa manifold, for which they were introduced recently by the first-named author, to the general case of page-$1$-$\partial\bar\partial$-manifolds that were jointly introduced very recently by all three authors. We go on to obtain an analogue of the unobstructedness theorem of Bogomolov, Tian and Todorov for Calabi-Yau page-$1$-$\partial\bar\partial$-manifolds. As applications of this discussion, we study the small deformations of certain Nakamura solvmanifolds and reinterpret the cases of the Iwasawa manifold and its $5$-dimensional analogue from this standpoint.}

\vspace{2ex}

\section{Introduction}\label{section:Introduction} In this paper, we begin to investigate the role that the new class of page-$1$-$\partial\bar\partial$-manifolds introduced in [PSU20a] plays in the theory of {\bf deformations} of complex structures and in the new approach to {\bf Mirror Symmetry}, extended to the possibly non-K\"ahler context, proposed in [Pop18].

\vspace{2ex}

(I)\, On the one hand, we introduce in $\S.$\ref{section:page-1_essential-deformations} the notion of {\bf small essential deformations} of an arbitrary compact {\it Calabi-Yau page-$1$-$\partial\bar\partial$-manifold}.

\vspace{1ex}

The special case of the $3$-dimensional Iwasawa manifold $I^{(3)}$ (a complex parallelisable nilmanifold that is also a Calabi-Yau page-$1$-$\partial\bar\partial$-manifold) was treated in [Pop18]. Recall that a compact complex manifold $X$ is said to be {\it complex parallelisable} if its holomorphic tangent bundle $T^{1,\,0}X$ is {\it trivial}. By Wang's theorem [Wan54], any complex parallelisable $X$ is the quotient $G/\Gamma$ of a {\it complex} Lie group $G$ by a co-compact, discrete subgroup $\Gamma$. When the {\it nilpotent} Lie group $G$ is merely a real Lie group endowed with a left-invariant complex structure, $X=G/\Gamma$ is said to be a {\it nilmanifold}. Recall that every nilmanifold $X$ is a {\it Calabi-Yau manifold} in the sense (adopted in this paper) that its canonical line bundle $K_X$ is {\it trivial}. In the case of $I^{(3)}$, $G$ is a {\it nilpotent complex} Lie group, the Heisenberg group of $3\times 3$ upper-triangular matrices with entries in $\C$.

The {\it small essential deformations} of the Iwasawa manifold $I^{(3)}$ were defined in [Pop18] as those small deformations of $I^{(3)}$ that are {\it not complex parallelisable}.




According to [Pop18], after removing the complex parallelisable small deformations of the $3$-dimensional Iwasawa manifold $X=I^{(3)}$ from its Kuranishi family, the remaining, {\it essential}, small deformations turn out to be parametrised by the $E_2$-cohomology space $E_2^{2,\,1}(X)$ featuring on the second page of the {\it Fr\"olicher spectral sequence (FSS)} of $I^{(3)}$. Recall that the ordinary small deformations are parametrised by the bigger-dimensional Dolbeault cohomology space $H^{2,\,1}_{\bar\partial}(X)=E_1^{2,\,1}(X)$ featuring on the first page of the FSS. The small essential deformations of $I^{(3)}$ have much better Hodge-theoretical properties than the ordinary small deformations. Indeed, the FSS of $I^{(3)}$ degenerates at $E_2$, rather than at $E_1$, and an analogue of the Hodge decomposition and symmetry for $I^{(3)}$ was observed in [Pop18] when the traditional first page is replaced by the second page of the FSS.

This phenomenon was generalised in [PSU20a] through the introduction of the new class of {\it page-$1$-$\partial\bar\partial$-manifolds} (to which the Iwasawa manifolds belong) and, more generally, that of {\it page-$r$-$\partial\bar\partial$-manifolds} for every non-negative integer $r$.




It turns out that for some complex parallelisable $n$-dimensional nilmanifolds $X$, such as the $5$-dimensional analogue $I^{(5)}$ of the Iwasawa manifold $I^{(3)}$, the non-complex parallelisable small deformations no longer coincide with those parametrised by $E_2^{n-1,\,1}(X)$ even when this space injects in a natural way into $E_1^{n-1,\,1}(X)$. Recall that, classically, $E_1^{n-1,\,1}(X)$ parametrises all the small deformations of $X$ when there is no obstruction to deforming the complex structure of $X$.

\vspace{2ex}

We define the {\bf small essential deformations} of a compact Calabi-Yau page-$1$-$\partial\bar\partial$-manifold $X$ as those small deformations of $X$ that are parametrised by $E_2^{n-1,\,1}(X)$. However, in general, there is no canonical injection of $E_2^{n-1,\,1}(X)$ into $E_1^{n-1,\,1}(X)$, although there is always a {\it canonical surjection} from a canonical vector subspace $E_1^{n-1,\,1}(X)_0$ of $E_1^{n-1,\,1}(X)$ to $E_2^{n-1,\,1}(X)$. Under the {\it Calabi-Yau page-$1$-$\partial\bar\partial$}-assumption on $X$, we manage to associate in a natural, unique way a linear {\it injection}\begin{equation*}\label{eqn:omega_lift-map}J^{n-1,\,1}_\omega:E_2^{n-1,\,1}(X)\hooklongrightarrow E_1^{n-1,\,1}(X)_0\subset E_1^{n-1,\,1}(X)\end{equation*} with every Hermitian metric $\omega$ on $X$ such that $J^{n-1,\,1}_\omega$ is a lift (i.e. a section) of the canonical surjection $E_1^{n-1,\,1}(X)_0\twoheadrightarrow E_2^{n-1,\,1}(X)$. (See Conclusion \ref{Conc:minimal-d-closed_rep} and the discussion in $\S.$\ref{section:page-1_essential-deformations} leading up to it.)

As a consequence of this discussion, we propose the following definition of small essential deformations of the complex structure of $X$ induced by a given metric $\omega$ on $X$.

 \begin{Def}\label{Def:introd_ess-deformations} Let $X$ be a compact complex $n$-dimensional {\bf Calabi-Yau page-$1$-$\partial\bar\partial$-manifold}.

 For any Hermitian metric $\omega$ on $X$, the space $E_1^{n-1,\,1}(X)_{\omega\mbox{-ess}}$ of {\bf small $\omega$-essential deformations} of $X$ is defined as the image of $J^{n-1,\,1}_{\omega}$: \[E_1^{n-1,\,1}(X)_{\omega\mbox{-ess}}:= J^{n-1,\,1}_{\omega}(E_2^{n-1,\,1}(X))\subset E_1^{n-1,\,1}(X).\]

 \end{Def}

 See Definition \ref{Def:ess-deformations} for further details. In particular, when $X$ has a canonical (i.e. depending only on its complex structure) Hermitian metric $\omega_0$, the small $\omega_0$-essential deformations of $X$ will be called the {\it small essential deformations}.


\vspace{2ex}

(II)\, On the other hand, we study in $\S.$\ref{section:page-1-unobstructedness} the possible {\bf unobstructedness} of the small deformations, both essential and ordinary, of {\bf Calabi-Yau page-$1$-$\partial\bar\partial$}-complex structures. We get a generalisation of the following classical Bogomolov-Tian-Todorov theorem (see [Tia87], [Tod89]): \\

 {\it The Kuranishi family of a compact {\bf K\"ahler Calabi-Yau} manifold is {\bf unobstructed}.} \\

 \noindent The K\"ahler assumption can be weakened to the $\partial\bar\partial$-assumption (and even to the $E_1=E_\infty$ assumption), as pointed out in [Tia87], [Ran92] and [Kaw92] -- see also discussions by various other authors such as [Pop13, Theorem 1.2].

Our main result is the following statement to the effect that, under certain cohomological conditions, a similar phenomenon holds when the $\partial\bar\partial$-assumption is further weakened to the page-$1$-$\partial\bar\partial$-assumption. See Definition \ref{Def:ess_unobstructedness} for the meaning of {\it unobstructedness} for the {\it essential Kuranishi family}. Meanwhile, for any bidegree $(p,\,q)$, we let ${\cal Z}_r^{p,\,q}(X)$ stand for the vector space of smooth $E_r$-closed $(p,\,q)$-forms on $X$. (These are the smooth $(p,\,q)$-forms on $X$ that represent $E_r$-cohomology classes on the $r$-th page of the Fr\"olicher spectral sequence. See e.g. Proposition 2.3 in [PSU20b], reproduced as Proposition \ref{Prop:E_r-closed-exact_conditions} below, for a description of them.)

\begin{The}\label{The:Introd_page-1-ddbar_unobstructedness} Let $X$ be a compact {\bf Calabi-Yau page-$1$-$\partial\bar\partial$-manifold} with $\mbox{dim}_\C X=n$. Fix a non-vanishing holomorphic $(n,\,0)$-form $u$ on $X$ and suppose that

\begin{equation}\label{eqn:extra-hyp_unobstructedness}\psi_1(t)\lrcorner(\rho_1(s)\lrcorner u)\in{\cal Z}_2^{n-2,\,2}\end{equation}

\noindent for all $\psi_1(t), \rho_1(s)\in C^\infty_{0,\,1}(X,\,T^{1,\,0}X)$ such that $\psi_1(t)\lrcorner u, \rho_1(s)\lrcorner u\in\ker d\cup\mbox{Im}\,\partial$.

\vspace{1ex}

(i)\, Then, the {\bf essential Kuranishi family} of $X$ is {\bf unobstructed}.

\vspace{1ex}

(ii)\, If, moreover, ${\cal Z}_1^{n-1,\,1} = {\cal Z}_2^{n-1,\,1}$, the {\bf Kuranishi family} of $X$ is {\bf unobstructed}.

\end{The}

This undertaking is justified by the fact that unobstructedness of the Kuranishi family occurs for some well-known compact complex manifolds that are not $\partial\bar\partial$-manifolds (so, the Bogomolov-Tian-Todorov theorem may not apply), but are page-$1$-$\partial\bar\partial$-manifolds, such as $I^{(3)}$ and $I^{(5)}$. The point we will make is that $I^{(3)}$ and $I^{(5)}$ are not isolated examples, but they are part of a pattern.

\vspace{2ex}

(III)\, Finally, we demonstrate in $\S.$\ref{section:explicit-computations} the role played by the {\it small essential deformations} by means of examples. Our main observation is the following

\begin{Prop}\label{Prop:obstructed-unobstructed_ess-non-ess} There exist compact Calabi-Yau page-$1$-$\partial\bar\partial$-manifolds with {\bf obstructed small deformations} but with {\bf unobstructed small essential deformations}.

\end{Prop}

The examples we exhibit with the above property are the {\bf Nakamura solvmanifolds of class (3b)} introduced in [Nak75] and taken up again in [Has10, $\S.3$] and [AK17].

\section{Background and preliminaries}\label{subsection:background}

We first review some general issues in the theory of deformations of complex structures, make an observation about some non-essential deformations and then review some specific properties of our class of page-$1$-$\partial\bar\partial$-manifolds recently introduced in [PSU20a].


\subsection{Small deformations of complex structures}\label{subsection:background-def}

Let $X$ be a compact complex manifold with $\mbox{dim}_\C X=n$. Recall that small deformations of the complex structure of $X$ over a base $B$ may be described by smooth $T^{1,\,0}X$-valued $(0,\,1)$-forms $\psi(t)\in C^\infty_{0,\,1}(X,\,T^{1,\,0}X)$ depending on $t\in B$ and satisfying the {\it integrability condition} (see e.g. [KS60]):

\begin{equation}\label{eqn:integrability-cond}\bar\partial\psi(t) = \frac{1}{2}\,[\psi(t),\,\psi(t)].\end{equation}

\noindent In fact, given such a $\psi$, the space of $(0,1)$-tangent vectors for the complex structure determined by $\psi$ is given by $(\Id+\psi)T^{0,1}_X$.

Let $t=(t_1,...,t_m)\in\C^m$ with $m=\mbox{dim}_\C H^{0,\,1}(X,\,T^{1,\,0}X)$. Writing \[
\psi(t) = \psi_1(t) + \sum\limits_{\nu=2}^{+\infty}\psi_\nu(t)\]

for the Taylor expansion of $\psi$ around $0$, (so each $\psi_\nu(t)$ is a homogeneous polynomial of degree $\nu$ in the variables $t=(t_1,\dots , t_m)$), the integrability condition is easily seen to be equivalent to $\delbar\psi_1(t)=0$ and the following sequence of conditions:

\begin{equation*}\label{eqn:integrability-condition_nu}\bar\partial\psi_\nu(t) = \frac{1}{2}\,\sum\limits_{\mu=1}^{\nu-1}[\psi_\mu(t),\,\psi_{\nu-\mu}(t)] \hspace{3ex} (\mbox{Eq.}\,\,(\nu)), \hspace{3ex} \nu\geq 2. \end{equation*}

\noindent The {\bf Kuranishi family} of $X$ is said to be {\bf unobstructed} if there exists a choice $\{\beta_1,\dots , \beta_m\}$ of representatives of cohomology classes that form a basis $\{[\beta_1],\dots , [\beta_m]\}$ of $H^{0,\,1}(X,\,T^{1,\,0}X)$ such that the integrability condition is satisfied (i.e. all the equations (Eq.\,$(\nu)$) are solvable) for any choice of parameters $(t_1,\dots , t_m)\in\C^m$ defining $\psi_1(t)=t_1\beta_1 + \dots + t_m\beta_m$.

By the fundamental result of [Kur62], when this qualitative condition is satisfied, a convergent solution $\psi(t)$ can be built for small $t$ through an inductive construction of the $\psi_\nu(t)$'s from the given $\psi_1(t)$ by solving the equations (Eq.\,$(\nu)$) and choosing at every step the solution with minimal $L^2$ norm for a pregiven Hermitian metric on $X$. The r.h.s. of each of these equations is $\bar\partial$-closed, so the only obstruction to solvability is the possible non-$\bar\partial$-exactness of the r.h.s. The resulting (germ of a) family $(X_t)_{t\in\Delta}$ of complex structures on $X$ is called the {\it Kuranishi family} of $X$. (It depends on the metric, but different choices of metrics yield isomorphic families.) If it is unobstructed, its base $B$ is smooth and can be viewed as an open ball about $0$ in the cohomology vector space $H^{0,\,1}(X,\,T^{1,\,0}X)$.

If, moreover, the canonical bundle $K_X$ of $X$ is {\bf trivial} (and we call $X$ a {\bf Calabi-Yau manifold} in this case), there exists a (unique up to scalar multiplication) smooth non-vanishing holomorphic $(n,\,0)$-form $u$ on $X$. It induces isomorphisms: \begin{eqnarray}\label{eqn:C-Y_isom_forms}C^{\infty}_{0, \, 1}(X, \, T^{1, \, 0}X)\ni\theta\mapsto\theta\lrcorner u\in C^{\infty}_{n-1,\, 1}(X, \,\C),\end{eqnarray} \begin{eqnarray}\label{eqn:C-Y_isom_cohom}H^{0,\,1}(X,\,T^{1,\,0}X)\ni[\theta]\mapsto [\theta\lrcorner u]\in H^{n-1,\,1}_{\bar\partial}(X)=E_1^{n-1,\,1}(X),\end{eqnarray}

  \noindent that we call the {\bf Calabi-Yau isomorphisms} (on forms, resp. in cohomology), where the operation denoted by $\cdot\lrcorner$ combines the contraction of $u$ by the $(1, \, 0)$-vector field component of $\theta$ with the exterior multiplication by the $(0,\,1)$-form component. In particular, if the Kuranishi family of a Calabi-Yau manifold $X$ is unobstructed, its base $B$ can be viewed as an open ball about $0$ in $H^{n-1,\,1}_{\bar\partial}(X)$.

\begin{Ex}\label{Ex:Kuranishi_I5} $(${\bf The Kuranishi family of the $5$-dimensional Iwasawa-type manifold}$)$

\vspace{1ex}

 {\rm Let us now consider the specific example of the complex parallelisable nilmanifold $X=I^{(5)}$ of complex dimension $5$. Its complex structure is described by five holomorphic $(1,\,0)$-forms $\varphi_1,\dots , \varphi_5$ satisfying the equations: \begin{equation*}\label{eqn:varphi_j-equations}d\varphi_1 = d\varphi_2 = 0, \hspace{2ex} d\varphi_3 = \varphi_1\wedge\varphi_2, \hspace{2ex} d\varphi_4 = \varphi_1\wedge\varphi_3, \hspace{2ex} d\varphi_5 = \varphi_2\wedge\varphi_3.\end{equation*}

\noindent If $\theta_1, \dots , \,\theta_5$ form the dual basis of $(1,\,0)$-vector fields, then $[\theta_i,\,\theta_j]=0$ except in the following cases: \begin{eqnarray*}\label{eqn:theta_j-equations} [\theta_1,\,\theta_2] =  -\theta_3, \hspace{3ex}   [\theta_1,\,\theta_3]  =  -\theta_4,  \hspace{3ex}  [\theta_2,\,\theta_3] = -\theta_5,\end{eqnarray*}

\noindent hence also \hspace{11ex} $[\theta_2,\,\theta_1] =  \theta_3,  \hspace{3ex}  [\theta_3,\,\theta_1]  =  \theta_4,  \hspace{3ex} [\theta_3,\,\theta_2] = \theta_5.$

\vspace{2ex}

\noindent In particular, \hspace{2ex} $H^{0,\,1}(X,\,T^{1.\,0}X) = \langle[\overline\varphi_1\otimes\theta_i],\,[\overline\varphi_2\otimes\theta_i]\,\mid\,i=1,\dots , 5\rangle$, so $\mbox{dim}_\C H^{0,\,1}(X,\,T^{1.\,0}X)=10$. }

\end{Ex}

\vspace{2ex}

 This manifold is the $5$-dimensional analogue of the $3$-dimensional Iwasawa manifold $I^{(3)}$. The following fact was observed in [Rol11].

\begin{Prop}\label{Prop:unobstructedness} The Kuranishi family of the $5$-dimensional nilmanifold $I^{(5)}$ is {\bf unobstructed}.

\end{Prop}

\noindent {\it Proof.} It was given in [Rol11].  \hfill $\Box$

\subsection{Cohomological triviality of complex parallelisable deformations of nilmanifolds}\label{subsection:cohomological triviality}

In this subsection, we pave the way for the later introduction in $\S.$\ref{section:page-1_essential-deformations} of the notion of {\it small essential deformations} by exhibiting the opposite type of objects: a class of superfluous (hence dispensable) small deformations. Specifically, we show that the {\it complex parallelisable} small deformations of a {\it complex parallelisable nilmanifold} have the same geometry (from the cohomological point of view -- see Corollary \ref{Cor:small-def_cohom-isom} -- and as far as the universal cover is concerned -- see Theorem \ref{The:Introd_universal-covers_def}) as the original manifold.

For a complex parallelisable nilmanifold $X=G/\Gamma$, where $G$ is a simply connected nilpotent complex Lie group and $\Gamma\subset G$ is a lattice, the Dolbeault cohomology can be computed by left invariant forms (cf. [Sak76]). In particular, one has (cf. [Nak75]):

\[H^{0,1}(X,T^{1,0}X)\cong H^{0,\,1}(X,\C)\otimes \fg^{1,0}=(\ker\delbar\cap A_\fg^{0,1})\otimes \fg^{1,0},\] where $\fg$ is the Lie algebra of $G$.

Furthermore, $\fg$ is actually a {\it complex} Lie algebra and $\fg^{1,0}\subset \fg_\C$ is a {\it complex} subalgebra, where $\fg_\C$ is the complexification of $\fg$. In fact, one has an identification of complex Lie algebras $\fg\cong \fg^{1,0}$ given by $z\mapsto \frac 1 2 (z-iJz)$. In what follows, we will always tacitly use the above identifications.

Of particular interest are the cohomology classes in \[
H_{par}(X):=H^{0,\,1}(X,\C)\otimes Z(\fg)=(\ker\delbar\cap A^{0,1}_\fg)\otimes Z(\fg)\subset H^{0,1}(X,T^{1,0}X),
\]

where $Z(\fg)$ is the centre of $\fg$, (which coincides with the Lie algebra of the centre $Z(G)$ of $G$ since $G$ is connected). The last inclusion is a consequence of the identification $\fg\cong\fg^{1,0}$.  Elements in $H_{par}(X)$ are called {\it infinitesimally complex parallelisable deformations} of $X$ due to the following

\begin{The}([Rol11]) Let $X=G/\Gamma$ be a complex parallelisable nilmanifold. Let $\mu\in H^{0,1}(X,T^{1,0}X)$. The following statements are equivalent.

  \vspace{1ex}

  $(1)$\, $\mu\in H_{par}(X)$.

  \vspace{1ex}

  $(2)$\, For all $X,Y\in \fg$, one has $[X,\mu \overline Y]=0$.

  \vspace{1ex}

$(3)$\, $t\mu$ induces a $1$-parameter family of complex parallelisable manifolds for $t$ small enough.

\vspace{1ex}

Moreover, for each $\mu$ satisfying $(1)$, $(2)$ or $(3)$, the sequence of equations $(\mbox{Eq.}\,\,(\nu))_{\nu\geq 1}$ (equivalently, (\ref{eqn:integrability-cond})) is solvable with $\psi=\psi_1=\mu$.

\end{The}

We will show that the cohomology is the same for all the complex parallelisable small deformations of a given complex parallelisable nilmanifold $X=G/\Gamma$. This will be a consequence of the following

\begin{The}\label{The:Introd_universal-covers_def} Let $X=G/\Gamma$ be a {\bf complex parallelisable nilmanifold}, where $G$ is a simply connected nilpotent complex Lie group and $\Gamma\subset G$ is a lattice. The universal cover of any complex parallelisable small deformation of $X$ is isomorphic to $G$ as a complex Lie groups. 

\end{The}

\noindent{\it Proof.} It is known that, for any left-invariant complex structure $J$ making $X$ into a complex parallelisable nilmanifold $X=G/\Gamma$, all the small deformations of $J$ are again left-invariant (cf. [Rol11, sect. 4]). In particular, they are again of the form $G/\Gamma$, but now with a possibly different, yet still left invariant, complex structure. Thus, differentiably, the universal cover is always $G$, which is determined entirely by $\fg$ through the Lie-group/Lie-algebra correspondence. However, the complex structure on $G$ varies with $\mu$ but since it remains left-invariant, it is determined by the splitting $\fg_\C=\fg^{0,1}_\mu\oplus\fg^{1,0}_\mu$ into $i$- and $(-i)$-eigenspaces, which can be computed from the complex structure of the central fibre via the equalities $\fg^{0,1}_\mu=(\Id+\mu)\fg^{0,1}$ and $\fg^{1,0}_\mu=(\Id+\bar\mu)\fg^{1,0}$.

{\begin{Claim}\label{Claim:Lie-alg-isom} The linear map of vector spaces

\[
\alpha: \fg_\C\longrightarrow \fg_\C,
\]

defined as $(\Id+\mu)$ on $\fg^{0,1}_{0}$ and as $(\Id+\overline{\mu})$ on $\fg^{1,0}_{0}$, is an isomorphism of Lie algebras.

\end{Claim}

  \noindent {\it Proof of Claim \ref{Claim:Lie-alg-isom}.} Since $\mu$ is small, $\alpha$ is an isomorphism of vector spaces and the point is to show that it is also a morphism of Lie algebras. We use $[X,\bar Y]=0$ for all $X\in \fg^{1,0}$ and $\bar Y\in \fg^{0,1}$. Since $\mu\in H^{0,\,1}(X,\C)\otimes Z(\fg)$, one also has $[X,\mu\bar Y]=0$, so $[Z,\mu \bar Y]=[Z,\bar\mu X]=0$ for any $Z\in \fg_\C$. So, for $\bar X,\bar Y\in \fg^{0,1}$, we have:  \begin{eqnarray*}[\alpha \bar X,\alpha \bar Y] &= & [\bar X,\bar Y]+[\mu\bar X,\mu \bar Y] +[\mu \bar X,\bar Y]+[\bar X,\mu\bar Y] \\
& = & [\bar X, \bar Y] =[\bar X,\bar Y]+ \mu([\bar X,\bar Y]) =\alpha([\bar X,\bar Y]).\\
\end{eqnarray*} Regarding the last-but-one equality, recall Cartan's formula $(\delbar\bar\eta)(\bar X,\bar Y)=-\bar\eta([\bar X,\bar Y])$ that holds for any left-invariant $(0,1)$-form $\bar\eta$ and that $\mu\in \ker\delbar\cap A^{0,1}_{\fg}\otimes Z(\fg)$. Therefore, $\mu([\bar X, \bar Y])=-(\delbar\mu)(\bar X,\bar Y)=0$. By a similar argument, $[\alpha X,\alpha Y]=\alpha([X,Y])$ for $X,Y\in\fg^{1,0}$.

Finally, for all $X\in\fg^{1,\,0}$ and all $\bar Y\in \fg^{0,1}$, we have: \begin{eqnarray*}
[\alpha X,\alpha \bar Y] &= & [X,\bar Y] + [\bar\mu X,\mu\bar Y] +[\bar\mu X,\bar Y]+[X,\mu\bar Y]\\
&= & 0 = \alpha([X,\bar Y]).
\end{eqnarray*}

Summing up, $\alpha$ is an isomorphism of Lie algebras. Thus, we get an induced isomorphism $G\rightarrow G$ which by construction is compatible with the complex structures corresponding to $0$ resp. $\mu$.

This finishes the proof of Claim \ref{Claim:Lie-alg-isom} and that of Theorem \ref{The:Introd_universal-covers_def}.

\hfill $\Box$

\begin{Rem} Note that Theorem \ref{The:Introd_universal-covers_def} does \textbf{not} state that the corresponding small deformations of $X$ are themselves biholomorphic. For example, when $X$ is a torus, we only recover the fact that the universal cover of each small deformation is $\C^n$ (while, of course, the lattice changes).

\end{Rem}

For the last two cohomologies mentioned in the following statement, see Definition 3.4. in [PSU20b], recalled as Definition \ref{Def:E_r-BC_E_r-A} below.

\begin{Cor}\label{Cor:small-def_cohom-isom} Let $X'$ be a complex parallelisable {\bf small deformation} of a complex parallelisable nilmanifold $X$. Then, there exists an {\bf isomorphism} between the double complexes of left invariant forms on $X$ and $X'$.

  In particular, there exist {\bf isomorphisms} $H(X)\cong H(X')$, where $H$ stands for any cohomology of one of the following types: Dolbeault, Fr\"olicher $E_r$,  De Rham, Bott-Chern, Aeppli and higher-page Bott-Chern and Aeppli.

\end{Cor}

\begin{proof} The first statement follows from Claim \ref{Claim:Lie-alg-isom}, since the double complex of left invariant forms can be computed in terms of the Lie-algebra with its complex structure, while the second follows from [Ste21, Prop. 12] and  the fact that for any nilmanifold $X=G/\Gamma$, the inclusion of the double complex of left-invariant forms on $G$ into all forms on $X$ is an {\bf $E_1$-isomorphism}. (This is conjectured to hold for all complex nilmanifolds and it is known for complex parallelisable ones, see [Sak76]).

\end{proof}

\subsection{Brief review of page-$r$-$\partial\bar\partial$-manifolds}\label{subsection:review_page-r-ddbar}

Let $X$ be a compact complex manifold with $\mbox{dim}_\C X=n$.

\subsubsection{$\partial\bar\partial$-{\bf manifolds}}

Recall that $X$ is said to be a $\partial\bar\partial$-{\bf manifold} if for any $d$-closed {\it pure-type} form $u$ on $X$, the following exactness properties are equivalent: \\

\hspace{10ex} $u$ is $d$-exact $\Longleftrightarrow$ $u$ is $\partial$-exact $\Longleftrightarrow$ $u$ is $\bar\partial$-exact $\Longleftrightarrow$ $u$ is $\partial\bar\partial$-exact.

\vspace{1ex}

The $\partial\bar\partial$-property of $X$ is equivalent to $X$ admitting a canonical {\it Hodge decomposition}. This implies that $X$ admits a canonical {\it Hodge symmetry} as well. These are properties of the Dolbeault cohomology groups $H^{p,\,q}_{\bar\partial}(X,\,\C) = E_1^{p,\,q}(X)$ of $X$. They lie on the first page of the Fr\"olicher spectral sequence (FSS) of $X$. The $\partial\bar\partial$-property of $X$ can also be characterised in terms of the Bott-Chern and Aeppli cohomologies. (See Theorem and Definition \ref{The-Def:main-def-prop_introd} below in the case $r=1$ for precise statements of these properties.)

\subsubsection{Terminology}

Recall that the {\bf Fr\"olicher spectral sequence (FSS)} of $X$ is a finite collection of complexes (called {\it pages}) that inductively refine the Dolbeault cohomology of $X$ until it becomes (non-canonically) isomorphic to the De Rham cohomology. The first page, whose spaces are denoted by $E_1^{\bullet,\,\bullet}$, is defined as the {\it Dolbeault complex} (with spaces $E_1^{p,\,\bullet}$ of bidegrees $(p,\,\bullet)$ for every fixed $p$), while for every integer $r\geq 2$ the $r$-th page $E_r^{\bullet,\,\bullet}$ is defined as the cohomology of the previous page.

In [PSU20b, $\S.2.2$], we introduced the following terminology. Fix $r\in\N$ and a bidegree $(p,\,q)$ with $p,q\in\{0,\dots, n\}$. A smooth $\C$-valued $(p,\,q)$-form $\alpha$ on $X$ is said to be {\bf $E_r$-closed} if it represents an $E_r$-cohomology class, denoted by $\{\alpha\}_{E_r}\in E_r^{p,\,q}(X)$, on the $r$-th page of the Fr\"olicher spectral sequence of $X$. Meanwhile, $\alpha$ is said to be {\bf $E_r$-exact} if it represents the {\it zero} $E_r$-cohomology class, i.e. if $\{\alpha\}_{E_r}=0\in E_r^{p,\,q}(X)$. The $\C$-vector space of $C^\infty$ $E_r$-closed (resp. $E_r$-exact) $(p,\,q)$-forms will be denoted by ${\cal Z}_r^{p,\,q}(X)$ (resp. ${\mathscr C}_r^{p,\,q}(X)$). Of course, ${\mathscr C}_r^{p,\,q}(X)\subset{\cal Z}_r^{p,\,q}(X)$ and $E_r^{p,\,q}(X) = {\cal Z}_r^{p,\,q}(X)/{\mathscr C}_r^{p,\,q}(X)$.

These notions are explicitly characterised as follows. (See [PSU20b, Proposition $2.3.$].)

\begin{Prop}\label{Prop:E_r-closed-exact_conditions} (i)\, Fix $r\geq 2$. A form $\alpha\in C^\infty_{p,\,q}(X)$ is {\bf $E_r$-closed} if and only if there exist forms $u_l\in C^\infty_{p+l,\,q-l}(X)$ with $l\in\{1,\dots , r-1\}$ satisfying the following tower of $r$ equations: \begin{eqnarray*}\label{eqn:tower_E_r-closedness} \bar\partial\alpha & = & 0 \\
     \partial\alpha & = & \bar\partial u_1 \\
     \partial u_1 & = & \bar\partial u_2 \\
     \vdots & & \\
     \partial u_{r-2} & = & \bar\partial u_{r-1}.\end{eqnarray*}

We say in this case that $\bar\partial\alpha=0$ and $\partial\alpha$ {\bf runs at least $(r-1)$ times}.

\vspace{1ex}

  (ii)\, Fix $r\geq 2$. A form $\alpha\in C^\infty_{p,\,q}(X)$ is {\bf $E_r$-exact} if and only if there exist forms $\zeta\in C^\infty_{p-1,\,q}(X)$ and $\xi\in C^\infty_{p,\,q-1}(X)$ such that $$\alpha=\partial\zeta + \bar\partial\xi,$$

\noindent with $\xi$ arbitrary and $\zeta$ satisfying the following tower of $(r-1)$ equations: \begin{eqnarray}\label{eqn:tower_E_r-exactness_l}\nonumber \bar\partial\zeta & = & \partial v_{r-3}  \\
    \nonumber \bar\partial v_{r-3} & = & \partial v_{r-4} \\
    \nonumber \vdots & & \\
    \nonumber \bar\partial v_1 & = & \partial v_0 \\
    \nonumber           \bar\partial v_0 & = & 0,\end{eqnarray}

\noindent for some forms $v_0,\dots , v_{r-3}$. (When $r=2$, $\zeta_{r-2} = \zeta_0$ must be $\bar\partial$-closed.)

We say in this case that $\bar\partial\zeta$ {\bf reaches $0$ in at most $(r-1)$ steps}.

\vspace{1ex}

(iii)\, The following inclusions hold in every bidegree $(p,\,q)$:

$$\dots\subset{\mathscr C}_r^{p,\,q}(X)\subset{\mathscr C}_{r+1}^{p,\,q}(X)\subset\dots\subset{\cal Z}_{r+1}^{p,\,q}(X)\subset{\cal Z}_r^{p,\,q}(X)\subset\dots,$$

\noindent with $\{0\}={\mathscr C}_0^{p,\,q}(X)\subset{\mathscr C}_1^{p,\,q}(X)=(\mbox{Im}\,\bar\partial)^{p,\,q}$ and ${\cal Z}_1^{p,\,q}(X)=(\ker\bar\partial)^{p,\,q}\subset{\cal Z}_0^{p,\,q}(X)=C^\infty_{p,\,q}(X)$.

\end{Prop}

\vspace{2ex}

Moreover, we say that a form $\alpha\in C^\infty_{p,\,q}(X)$ is {\bf $\overline{E}_r$-closed} if $\bar\alpha$ is $E_r$-closed and we say that $\alpha$ is {\bf $\overline{E}_r$-exact} if $\bar\alpha$ is $E_r$-exact.

\vspace{2ex}

Taking our cue from Proposition \ref{Prop:E_r-closed-exact_conditions}, we defined in [PSU20b, Definition 3.1.] the following higher-page analogues of $\partial\bar\partial$-closedness and $\partial\bar\partial$-exactness.

\begin{Def}\label{Def:E_rE_r-bar} Suppose that $r\geq 2$.

\vspace{1ex}

(i)\, We say that a form $\alpha\in C^\infty_{p,\,q}(X)$ is {\bf $E_r\overline{E}_r$-closed} if there exist smooth forms $\eta_1,\dots , \eta_{r-1}$ and $\rho_1,\dots , \rho_{r-1}$ such that the following two towers of $r-1$ equations are satisfied: \begin{align*}\label{eqn:towers_E_rE_r-bar-closedness} \partial\alpha & = \bar\partial\eta_1 &    \bar\partial\alpha & = \partial\rho_1 & \\
    \partial\eta_1 & = \bar\partial\eta_2 &   \bar\partial\rho_1 & = \partial\rho_2 & \\
     \vdots & & \\
     \partial\eta_{r-2} & =  \bar\partial\eta_{r-1},  &  \bar\partial\rho_{r-2} & = \partial\rho_{r-1}. &\end{align*}

\vspace{1ex}

\vspace{1ex}

(ii)\, We refer to the properties of $\alpha$ in the two towers of $(r-1)$ equations under (i) by saying that $\partial\alpha$, resp. $\bar\partial\alpha$, {\bf runs at least $(r-1)$ times}.

(iii)\, We say that a form $\alpha\in C^\infty_{p,\,q}(X)$ is {\bf $E_r\overline{E}_r$-exact} if there exist smooth forms $\zeta, \xi, \eta$ such that \begin{equation}\label{eqn:main-eq_E_rE_r-bar-exactness}\alpha = \partial\zeta + \partial\bar\partial\xi + \bar\partial\eta\end{equation}

\noindent and such that $\zeta$ and $\eta$ further satisfy the following conditions. There exist smooth forms $v_{r-3},\dots , v_0$ and $u_{r-3},\dots , u_0$ such that the following two towers of $r-1$ equations are satisfied: \begin{align*}\label{eqn:towers_E_rE_r-bar-exactness} \bar\partial\zeta & = \partial v_{r-3} &    \partial\eta & = \bar\partial u_{r-3} & \\
     \bar\partial v_{r-3} & = \partial v_{r-4} &   \partial u_{r-3} & = \bar\partial u_{r-4} & \\
     \vdots & & \\
     \bar\partial v_0 & =  0,  &  \partial u_0 & = 0. &\end{align*}

\vspace{1ex}

(iv)\, We refer to the properties of $\zeta$, resp. $\eta$, in the two towers of $(r-1)$ equations under (iii) by saying that $\bar\partial\zeta$, resp. $\partial\eta$, {\bf reaches $0$ in at most $(r-1)$ steps}.

\end{Def}

 When $r-1=1$, the properties of $\bar\partial\zeta$, resp. $\partial\eta$, {\it reaching $0$ in $(r-1)$ steps} translate to $\bar\partial\zeta=0$, resp. $\partial\eta=0$.

To unify the definitions, we will also say that a form $\alpha\in C^\infty_{p,\,q}(X)$ is {\bf $E_1\overline{E}_1$-closed} (resp. {\bf $E_1\overline{E}_1$-exact}) if $\alpha$ is $\partial\bar\partial$-closed (resp. $\partial\bar\partial$-exact).

As with $E_r$ and $\overline{E}_r$, it follows at once from Definition \ref{Def:E_rE_r-bar} that the $E_r\overline{E}_r$-closedness condition becomes stronger and stronger as $r$ increases, while the $E_r\overline{E}_r$-exactness condition becomes weaker and weaker as $r$ increases. In other words, the following inclusions of vector spaces hold:

\vspace{2ex}

\noindent $\{\partial\bar\partial\mbox{-exact forms}\}\subset\dots\subset\{E_r\overline{E}_r\mbox{-exact forms}\}\subset\{E_{r+1}\overline{E}_{r+1}\mbox{-exact forms}\}\subset\dots $

\vspace{1ex}

\hfill $\dots\subset\{E_{r+1}\overline{E}_{r+1}\mbox{-closed forms}\}\subset\{E_r\overline{E}_r\mbox{-closed forms}\}\subset\dots\subset\{\partial\bar\partial\mbox{-closed forms}\}.$

\vspace{3ex}

The main notions introduced and studied in [PSU20b] are the following higher-page analogues of the Bott-Chern and Aeppli cohomologies.

\begin{Def}([PSU20b, Definition 3.4.])\label{Def:E_r-BC_E_r-A} Let $X$ be an $n$-dimensional compact complex manifold. Fix $r\in\N^\star$ and a bidegree $(p,\,q)$.

\vspace{1ex}

(i)\, The {\bf $E_r$-Bott-Chern} cohomology group of bidegree $(p,\,q)$ of $X$ is defined as the following quotient complex vector space: \[E_{r,\,BC}^{p,\,q}(X):=\frac{\{\alpha\in C^\infty_{p,\,q}(X)\,\mid\,d\alpha=0\}}{\{\alpha\in C^\infty_{p,\,q}(X)\,\mid\,\alpha\hspace{1ex}\mbox{is}\hspace{1ex} E_r\overline{E}_r\mbox{-exact}\}}.\]

\vspace{1ex}

(ii)\, The {\bf $E_r$-Aeppli} cohomology group of bidegree $(p,\,q)$ of $X$ is defined as the following quotient complex vector space: \[E_{r,\,A}^{p,\,q}(X):=\frac{\{\alpha\in C^\infty_{p,\,q}(X)\,\mid\,\alpha\hspace{1ex}\mbox{is}\hspace{1ex} E_r\overline{E}_r-\mbox{closed}\}}{\{\alpha\in C^\infty_{p,\,q}(X)\,\mid\,\alpha\in\mbox{Im}\,\partial + \mbox{Im}\,\bar\partial\}}.\]

\end{Def}

When $r=1$, the above groups coincide with the standard Bott-Chern, respectively Aeppli, cohomology groups. Moreover, for every $(p,\,q)$, one has a sequence of canonical linear {\bf surjections}: \begin{equation}\label{eqn:BC_seq-surjections}H_{BC}^{p,\, q}(X)=E_{1,\, BC}^{p,\, q}(X)\twoheadrightarrow E_{2,\, BC}^{p,\, q}(X)\twoheadrightarrow\dots\twoheadrightarrow E_{r,\, BC}^{p,\, q}(X)\twoheadrightarrow E_{r+1,\, BC}^{p,\, q}(X)\twoheadrightarrow\dots\end{equation} and a  sequence of {\bf subspaces} of $H_A^{p,\, q}(X)$: \begin{equation}\label{eqn:A_seq-inclusions}\dots\subset E_{r+1,A}^{p,\,q}(X)\subset E_{r,\, A}^{p,\,q}(X)\subset\dots\subset E_{1,A}^{p,\,q}(X)=H_A^{p,\, q}(X).\end{equation}

It can be shown that the representatives of $E_r$-Bott-Chern classes can be alternatively described as the forms that are simultaneously $E_r$-closed and $\overline{E}_r$-closed, while the $E_r$-Aeppli-exact forms can be alternatively described as those forms lying in ${\cal C}^{p,\,q}_r + \overline{\cal C}^{p,\,q}_r$.

\subsubsection{Higher-page analogues of the $\partial\bar\partial$-property}

In [PSU20a, Theorem and Definition 1.2.] and [PSU20b, Theorem 1.3.], we generalised the notion of $\partial\bar\partial$-manifold to the higher pages of the Fr\"olicher spectral sequence in the following form that also features the {\bf higher-page Bott-Chern cohomology} and the {\bf higher-page Aeppli cohomology} of $X$ introduced in [PSU20b].

\begin{The-Def}\label{The-Def:main-def-prop_introd} Let $X$ be a compact complex manifold with $\mbox{dim}_\C X=n$. Fix an arbitrary positive integer $r$. The following statements are equivalent.

\vspace{1ex}

$(1)$\, For every bidegree $(p,\,q)$, every class $\{\alpha^{p,\,q}\}_{E_r}\in E_r^{p,\,q}(X)$ can be represented by a {\bf $d$-closed $(p,\,q)$-form} and for every $k$, the linear map

$$\bigoplus_{p+q=k}E_r^{p,\,q}(X)\ni\sum\limits_{p+q=k}\{\alpha^{p,\,q}\}_{E_r}\mapsto\bigg\{\sum\limits_{p+q=k}\alpha^{p,\,q}\bigg\}_{DR}\in H^k_{DR}(X,\,\C)$$

\noindent is {\bf well-defined} by means of $d$-closed pure-type representatives and {\bf bijective}.

\vspace{1ex}

In this case, $X$ is said to have the {\bf $E_r$-Hodge Decomposition} property.

\vspace{2ex}

$(2)$\, The Fr\"olicher spectral sequence of $X$ {\bf degenerates at $E_r$} and the De Rham cohomology of $X$ is {\bf pure}.

\vspace{2ex}

$(3)$\, For all $p,q\in\{0,\dots , n\}$ and for every form $\alpha\in C^\infty_{p,\,q}(X)$ such that $d\alpha=0$, the following equivalences hold: \begin{equation*}\alpha\in\mbox{Im}\,d\iff\alpha\hspace{1ex}\mbox{is}\hspace{1ex}E_r\mbox{-exact}\iff\alpha\hspace{1ex}\mbox{is}\hspace{1ex}\overline{E}_r\mbox{-exact}\iff\alpha\hspace{1ex}\mbox{is}\hspace{1ex}E_r\overline{E}_r\mbox{-exact}.\end{equation*}

\vspace{2ex}

$(4)$\, For all $p,q\in\{0,\dots , n\}$, the canonical linear maps $$E_{r,\,BC}^{p,\,q}(X)\longrightarrow E_r^{p,\,q}(X) \hspace{3ex} \mbox{and} \hspace{3ex} E_r^{p,\,q}(X)\longrightarrow E_{r,\,A}^{p,\,q}(X)$$ are {\bf isomorphisms}, where $E_{r,\,BC}^{p,\,q}(X)$ and $E_{r,\,A}^{p,\,q}(X)$ are the {\bf $E_r$-Bott-Chern}, respectively the {\bf $E_r$-Aeppli}, cohomology groups of bidegree $(p,\,q)$ introduced in [PSU20b].

\vspace{2ex}

$(5)$\, For all $p,q\in\{0,\dots , n\}$, the canonical linear map $E_{r,\,BC}^{p,\,q}(X)\longrightarrow E_{r,\,A}^{p,\,q}(X)$ is {\bf injective}.

\vspace{2ex}

$(6)$\, One has $\mbox{dim} E_{r,\, BC}^k(X)=\mbox{dim} E_{r,\, A}^k(X)$ for all $k\in\{0,...,2n\}$.

\vspace{2ex}

A compact complex manifold X that satisfies any of the equivalent conditions $(1)$--$(5)$ is said to be a {\bf page-$(r-1)$-$\partial\bar\partial$-manifold}.

\end{The-Def}

\vspace{2ex}

The relations among these notions are captured in

\begin{Obs}([PSU20a, Corollary 2.11.)\label{Obs:increasing-r_page-r-ddbar} Let $X$ be a compact complex manifold.

  \vspace{1ex}

  (i)\, The following equivalence holds: \\

 \hspace{15ex} $X$ is a {\bf $\partial\bar\partial$-manifold} $\iff$ $X$ is a {\bf page-$0$-$\partial\bar\partial$-manifold};

  \vspace{1ex}

  (ii)\, For every $r\in\N^\star$, the following implication holds:

  \vspace{2ex}

\hspace{15ex} $X$ is a {\bf page-$r$-$\partial\bar\partial$-manifold} $\implies$ $X$ is a {\bf page-$(r+1)$-$\partial\bar\partial$-manifold}.

\end{Obs}

\vspace{2ex}

That the {\it $E_r$-Hodge Decomposition property} implies the higher-page analogue of the Hodge Symmetry property was proved in [PSU20a, Corollary 2.10.].

\begin{Cor}\label{Cor:E_r_decomp-sym} Let $r$ be a positive integer. Any page-$(r-1)$-$\partial\bar\partial$-manifold $X$ has the {\bf $E_r$-Hodge Symmetry} property in the following sense.

  \vspace{1ex}

  For all $p,q\in\{0,\dots , n\}$,

  \vspace{1ex}

  (a)\, every class $\{\alpha^{p,\,q}\}_{E_r}\in E_r^{p,\,q}(X)$ contains a {\bf $d$-closed representative} $\alpha^{p,\,q}$;

  (b)\, the linear map $$E_r^{p,\,q}(X)\ni\{\alpha^{p,\,q}\}_{E_r}\mapsto\overline{\{\overline{\alpha^{p,\,q}}\}_{E_r}}\in \overline{E_r^{q,\,p}(X)}$$

  \noindent is {\bf well-defined} (in the sense that it does not depend on the choice of $d$-closed representative $\alpha^{p,\,q}$ of the class $\{\alpha^{p,\,q}\}_{E_r}$) and {\bf bijective}.

\end{Cor}

\section{Essential deformations of Calabi-Yau manifolds}\label{section:page-1_essential-deformations}

The notion of {\it essential deformations} was introduced in [Pop18] in the special case of the Iwasawa manifold $I^{(3)}$. We now extend it to the class of all {\it Calabi-Yau page-$1$-$\partial\bar\partial$-manifolds} $X$. The idea is to keep only those small deformations of $X$ that are parametrised by $E_2^{n-1,\,1}(X)$ instead of the larger-dimensional $E_1^{n-1,\,1}(X)$. However, $E_2^{n-1,\,1}(X)$ need not inject canonically into $E_1^{n-1,\,1}(X)$, so we will have to work with injections defined by a background Hermitian metric on $X$.

Let $X$ be a compact complex manifold with $\mbox{dim}_\C X=n$. Recall that, for every integer $r\geq 1$ and every bidegree $(p,\,q)$, the vector space of smooth $E_r$-closed (resp. $E_r$-exact) $(p,\,q)$-forms on $X$ is denoted by ${\cal Z}_r^{p,\,q}(X)$ (resp. ${\cal C}_r^{p,\,q}(X)$). We now define the following vector subspace of $E_1^{p,\,q}(X)$: \begin{equation*}\label{eqn:E_1_0_def} E_1^{p,\,q}(X)_0:=\frac{\{\alpha\in C^\infty_{p,\,q}(X)\,\mid\,\bar\partial\alpha=0 \hspace{1ex}\mbox{and}\hspace{1ex} \partial\alpha\in\mbox{Im}\,\bar\partial\}}{\{\bar\partial\beta\,\mid\,\beta\in C^\infty_{p,\,q-1}(X)\}} = \frac{{\cal Z}_2^{p,\,q}(X)}{{\cal C}_1^{p,\,q}(X)}\subset E_1^{p,\,q}(X).\end{equation*}

\noindent In other words, $E_1^{p,\,q}(X)_0=\ker d_1$ consists of the $E_1$-cohomology classes (i.e. Dolbeault cohomology classes) representable by $E_2$-closed forms of type $(p,\,q)$, where $d_1:E_1^{p,\,q}(X)\longrightarrow E_1^{p+1,\,q}(X)$, $d_1([u]_{\bar\partial})=[\partial u]_{\bar\partial}$, is the differential in bidegree $(p,\,q)$ on the first page of the FSS.

\begin{Lem}\label{Lem:P-pq-map} For all $p,q$, the canonical linear map $$P^{p,\,q}:E_1^{p,\,q}(X)_0\to E_2^{p,\,q}(X), \hspace{3ex} \{\alpha\}_{E_1}\mapsto\{\alpha\}_{E_2},$$

  \noindent is well defined and {\bf surjective}. Its kernel consists of the $E_1$-cohomology classes representable by $E_2$-exact forms of type $(p,\,q)$.

  In particular, $P^{p,\,q}$ is injective (hence an isomorphism) if and only if ${\cal C}_1^{p,\,q}(X)={\cal C}_2^{p,\,q}(X) $.

\end{Lem}

\noindent {\it Proof.} Well-definedness means that $P^{p,\,q}(\{\alpha\}_{E_1})$ is independent of the choice of representative of the class $\{\alpha\}_{E_1}\in E_1^{p,\,q}(X)_0$. This follows from the inclusion ${\cal C}_1^{p,\,q}(X)\subset{\cal C}_2^{p,\,q}(X)$. The other three statements are obvious.  \hfill $\Box$

\vspace{2ex}

Since the map $P^{p,\,q}:E_1^{p,\,q}(X)_0\to E_2^{p,\,q}(X)$ is surjective, there exist injective linear maps $J^{p,\,q}:E_2^{p,\,q}(X)\to E_1^{p,\,q}(X)_0$ such that $P^{p,\,q}\circ J^{p,\,q} = \mbox{Id}_{E_2^{p,\,q}(X)}.$ However, there is no unique or even canonical choice for such a map $J^{p,\,q}$. Indeed, there may exist different representatives $\alpha_1$ and $\alpha_2$ of a same class $\{\alpha\}_{E_2}\in E_2^{p,\,q}(X)$ that represent distinct classes $\{\alpha_1\}_{E_1}\neq\{\alpha_2\}_{E_1}$ in $E_1^{p,\,q}(X)$ even if $\alpha_1$ and $\alpha_2$ are $d$-closed. Choosing a section $J^{p,\,q}:E_2^{p,\,q}(X)\to E_1^{p,\,q}(X)_0$ for $P^{p,\,q}$ amounts to choosing a representative $\alpha$ in each class $\{\alpha\}_{E_2}\in E_2^{p,\,q}(X)$.

\vspace{2ex}

Since we are concerned with small deformations of $X$, we assume henceforth that $(p,\,q)=(n,\,n-1)$, although most of the following arguments apply in any bidegree. Let us fix a Hermitian metric $\omega$ on $X$. By the Hodge theory for the $E_2$-cohomology introduced in [Pop16] (and used e.g. in [PSU20b]) and the standard Hodge theory for the Dolbeault cohomology, there are Hodge isomorphisms: $$E_2^{n-1,\,1}(X)\simeq{\cal H}_2^{n-1,\,1} = {\cal H}_{2,\,\omega}^{n-1,\,1} \hspace{3ex} \mbox{and} \hspace{3ex} E_1^{n-1,\,1}(X)\simeq{\cal H}_1^{n-1,\,1} = {\cal H}_{1,\,\omega}^{n-1,\,1}$$

\noindent associating with every $E_2$- (resp. $E_1$-)class its unique $E_2$- (resp. $E_1$-)harmonic representative (w.r.t. $\omega$), where the $\omega$-dependent harmonic spaces are defined by $${\cal H}_2^{n-1,\,1}:=\ker(\widetilde\Delta:C^\infty_{n-1,\,1}(X)\to C^\infty_{n-1,\,1}(X))\subset{\cal H}_1^{n-1,\,1}:=\ker(\Delta'':C^\infty_{n-1,\,1}(X)\to C^\infty_{n-1,\,1}(X))$$

\noindent and $\widetilde\Delta=\widetilde\Delta_\omega=\partial p''\partial^\star + \partial^\star p''\partial + \Delta''$ is the pseudo-differential Laplacian introduced in [Pop16], $\Delta''= \bar\partial\bar\partial^\star + \bar\partial^\star\bar\partial$ is the standard $\bar\partial$-Laplacian, both associated with the metric $\omega$, while $p''$ is the $L^2_\omega$-orthogonal projection onto $\ker\Delta''$.

\vspace{2ex}

Under our {\it page-$1$-$\partial\bar\partial$-assumption} on $X$, every class $\{\alpha\}_{E_2}\in E_2^{n-1,\,1}(X)$ can be represented by a $d$-closed $(n-1,\,1)$-form. However, the $\widetilde\Delta$-harmonic (= the minimal $L^2$-norm) representative need not be $d$-closed. Nevertheless, for a reason that will become apparent later on, we need to work with $d$-closed representatives. In order to make the choice of such a representative unique (once a metric $\omega$ has been fixed on $X$), we will modify the $\widetilde\Delta$-harmonic representative to a $d$-closed one in a unique way imposed by $\omega$.

We will need the following result from an earlier work in our current special case $r=2$.

\begin{The}([PSU20b, Theorem 4.3.])\label{The:page_r-1_ddbar-equivalence_bis} Let $X$ be a compact complex manifold with $\mbox{dim}_\C X=n$. Fix an arbitrary integer $r\geq 2$. The following properties are equivalent.

	\vspace{1ex}

	(A)\, $X$ is a {\bf page-$(r-1)$-$\partial\bar\partial$-manifold}.

	\vspace{1ex}

	(F)\, For all $p,q\in\{0,\dots , n\}$, the following identities of vector subspaces of $C^\infty_{p+1,\,q}(X)$ hold:

	\begin{equation*}\label{eqn:identities_ddbar_r_bis}(i)\hspace{2ex} \mbox{Im}\,(\partial\bar\partial) = \partial({\cal Z}_r^{p,\,q})  \hspace{3ex} \mbox{and} \hspace{3ex} (ii)\hspace{2ex} {\cal C}_r^{p,\,q}\cap\ker d = \mbox{Im}\,d.\end{equation*}

\end{The}

\vspace{2ex}

The preliminary observation that produces a unique $d$-closed representative of any class $\{\alpha\}_{E_2}\in E_2^{n-1,\,1}(X)$ naturally associated with a given metric is the following

\begin{Lem}\label{Lem:minimal-d-closed_rep} Let $X$ be a compact complex {\bf page-$1$-$\partial\bar\partial$-manifold} with $\mbox{dim}_\C X=n$. Let $\omega$ be an arbitrary Hermitian metric on $X$.

  Fix an arbitrary class $\{\alpha\}_{E_2}\in E_2^{n-1,\,1}(X)$ and let $\alpha_\omega$ be the $\widetilde\Delta_\omega$-harmonic representative of it.

\hspace{1ex}

(i)\, There exist solutions $\xi\in C^\infty_{n-1,\,0}(X,\,\C)$ of the equation $$(\star)\hspace{3ex}\partial\bar\partial\xi = -\partial\alpha_\omega.$$

(ii)\, For each solution $\xi$ of equation $(\star)$, $\alpha_\omega + \bar\partial\xi$ is a $d$-closed representative of $\{\alpha\}_{E_2}\in E_2^{n-1,\,1}(X)$.

(iii)\, The minimal $L^2_\omega$-norm solution of equation $(\star)$ is given by the formula \begin{eqnarray}\label{eqn:minimal-d-closed_rep}\xi_\omega = -(\partial\bar\partial)^\star\Delta_{BC}^{-1}(\partial\alpha_\omega),\end{eqnarray} where $\Delta_{BC}^{-1}$ is the Green operator of the Bott-Chern Laplacian $\Delta_{BC}$ induced by $\omega$.

\end{Lem}

\noindent {\it Proof.} (i) The $(n-1,\,1)$-form $\alpha_\omega$ represents an $E_2$-class, so it is $E_2$-closed (i.e. $\alpha_\omega\in{\cal Z}_2^{n-1,\,1}$). Therefore, the page-$1$-$\partial\bar\partial$-assumption on $X$ implies, thanks to (i) of property (F) in Theorem \ref{The:page_r-1_ddbar-equivalence_bis}, that $\partial\alpha_\omega\in\mbox{Im}\,(\partial\bar\partial)$. This amounts to equation $(\star)$ being solvable for $\xi$.

(ii) Since $\alpha_\omega$ is $E_2$-closed, $\bar\partial\alpha_\omega = 0$, hence $\bar\partial(\alpha_\omega + \bar\partial\xi) = 0$. Meanwhile, $\partial(\alpha_\omega + \bar\partial\xi) = 0$ because $\xi$ is a solution equation $(\star)$.

(iii) It is known (see e.g. [Pop15, Theorem 4.1.]) that, for any $\partial\bar\partial$-exact form $v$, the minimal $L^2$-norm solution of the equation $\partial\bar\partial u = v$ is given by the formula $u=(\partial\bar\partial)^\star\Delta_{BC}^{-1}v$. In our case, $v=-\partial\alpha_\omega$. \hfill $\Box$

\vspace{2ex}

Note that the $\widetilde\Delta_\omega$-harmonic hypothesis on $\alpha_\omega$ was not used in the above proof. It was made only to make $\alpha_\omega$ uniquely associated with the pair $(\omega,\,\{\alpha\}_{E_2})$. The upshot of the above construction is the following

\begin{Conc}\label{Conc:minimal-d-closed_rep} Let $X$ be a compact complex {\bf page-$1$-$\partial\bar\partial$-manifold} with $\mbox{dim}_\C X = n$. For any Hermitian metric $\omega$ on $X$, the $\omega$-dependent linear injection \begin{equation*}\label{eqn:omega_lift-map}J^{n-1,\,1}_\omega:E_2^{n-1,\,1}(X)\hooklongrightarrow E_1^{n-1,\,1}(X)_0\subset E_1^{n-1,\,1}(X), \hspace{3ex} J^{n-1,\,1}_\omega(\{\alpha\}_{E_2}):=\{\alpha_\omega - \bar\partial(\partial\bar\partial)^\star\Delta_{BC}^{-1}(\partial\alpha_\omega)\}_{E_1},\end{equation*} has the property \begin{equation*}\label{eqn:omega_lift-map_prop}P^{n-1,\,1}\circ J^{n-1,\,1}_\omega = \mbox{Id}_{E_2^{n-1,\,1}(X)},\end{equation*} where $P^{n-1,\,1}:E_1^{n-1,\,1}(X)_0\twoheadrightarrow E_2^{n-1,\,1}(X)$ is the canonical linear surjection introduced in Lemma \ref{Lem:P-pq-map}.

  The map $J^{n-1,\,1}_\omega$ is called the {\bf $\omega$-lift} of $P^{n-1,\,1}$.

\end{Conc}

 In particular, if a {\bf canonical metric} $\omega_0$ exists on $X$ (in the sense that $\omega_0$ depends only on the complex structure of $X$ with no arbitrary choices involved in its definition), the associated map $J^{n-1,\,1}_{\omega_0}$ constitutes a {\bf canonical injection} of $E_2^{n-1,\,1}(X)$ into $E_1^{n-1,\,1}(X)$.

 \begin{Def}\label{Def:ess-deformations} Let $X$ be a compact complex $n$-dimensional {\bf Calabi-Yau page-$1$-$\partial\bar\partial$-manifold}.

\vspace{1ex}

(i) For any Hermitian metric $\omega$ on $X$, the space of {\bf small $\omega$-essential deformations} of $X$ is defined as the image in $E_1^{n-1,\,1}(X)$ of the $\omega$-lift $J^{n-1,\,1}_{\omega}$ of $P^{n-1,\,1}$, namely \[E_1^{n-1,\,1}(X)_{\omega\mbox{-ess}}:= J^{n-1,\,1}_{\omega}(E_2^{n-1,\,1}(X))\subset E_1^{n-1,\,1}(X).\]

\vspace{1ex}

(ii) Suppose that $X$ carries a {\bf canonical Hermitian metric} $\omega_0$. The space of {\bf small essential deformations} of $X$ is defined as the image in $E_1^{n-1,\,1}(X)$ of the canonical injection $J^{n-1,\,1}_{\omega_0}$, namely \[E_1^{n-1,\,1}(X)_{ess}:= J^{n-1,\,1}_{\omega_0}(E_2^{n-1,\,1}(X))\subset E_1^{n-1,\,1}(X).\]

\end{Def}

\vspace{2ex}

If the page-$1$-$\partial\bar\partial$-assumption on $X$ is replaced by a more general one (for example, the page-$r$-$\partial\bar\partial$-assumption for some $r\geq 2$ or merely the $E_r(X)=E_\infty(X)$ assumption for a specific $r\geq 2$), one can define a version of essential deformations using higher pages than the second one. The most natural choice is the degenerating page $E_r=E_\infty$ of the FSS if $r>2$. Since at the moment we are mainly interested in page-$1$-$\del\delbar$-manifolds, we confine ourselves to $E_2$.

\begin{Ex}\label{Ex:Iwasawa_dim3_ess}(The {\bf Iwasawa manifold}) If $\alpha$, $\beta$, $\gamma$ are the three canonical holomorphic $(1,\,0)$-forms induced on the complex $3$-dimensional Iwasawa manifold $X=G/\Gamma$ by $dz_1, dz_2, dz_3-z_1\,dz_2$ from $\C^3$ (the underlying complex manifold of the Heisenberg group $G$), it is well known that $\alpha$ and $\beta$ are $d$-closed, while $d\gamma = \partial\gamma = -\alpha\wedge\beta\neq 0$. It is equally standard that the Dolbeault cohomology group of bidegree $(2,\,1)$ is generated as follows: $$E_1^{2,\,1}(X)=\bigg\langle [\alpha\wedge\gamma\wedge\overline{\alpha}]_{\bar\partial},\, [\alpha\wedge\gamma\wedge\overline{\beta}]_{\bar\partial},\, [\beta\wedge\gamma\wedge\overline{\alpha}]_{\bar\partial},\, [\beta\wedge\gamma\wedge\overline{\beta}]_{\bar\partial}\bigg\rangle \oplus  \bigg\langle [\alpha\wedge\beta\wedge\overline{\alpha}]_{\bar\partial},\, [\alpha\wedge\beta\wedge\overline{\beta}]_{\bar\partial}\bigg\rangle.$$

\noindent In particular, we see that every $E_1$-class of bidegree $(2,\,1)$ can be represented by a $d$-closed form. Since every pure-type $d$-closed form is also $E_2$-closed (and, indeed, $E_r$-closed for every $r$), we get $$E_1^{2,\,1}(X) = E_1^{2,\,1}(X)_0.$$

It is equally standard that the $E_2$-cohomology group of bidegree $(2,\,1)$ is generated as follows: $$E_2^{2,\,1}(X)=\bigg\langle [\alpha\wedge\gamma\wedge\overline{\alpha}]_{E_2},\, [\alpha\wedge\gamma\wedge\overline{\beta}]_{E_2},\, [\beta\wedge\gamma\wedge\overline{\alpha}]_{E_2},\, [\beta\wedge\gamma\wedge\overline{\beta}]_{E_2}\bigg\rangle.$$

\noindent It identifies canonically with the vector subspace $$H^{2,\,1}_{[\gamma]}(X)=\bigg\langle [\alpha\wedge\gamma\wedge\overline{\alpha}]_{\bar\partial},\, [\alpha\wedge\gamma\wedge\overline{\beta}]_{\bar\partial},\, [\beta\wedge\gamma\wedge\overline{\alpha}]_{\bar\partial},\, [\beta\wedge\gamma\wedge\overline{\beta}]_{\bar\partial}\bigg\rangle \simeq E_1^{2,\,1}(X)\bigg /\bigg\langle [\alpha\wedge\beta\wedge\bar\alpha]_{\bar\partial},
\, [\alpha\wedge\beta\wedge\bar\beta]_{\bar\partial}\bigg\rangle $$

\noindent of $E_1^{2,\,1}(X)$ introduced in [Pop18, $\S.4.2$] as parametrising the {\bf essential deformations} defined there for the Iwasawa manifold $X$.

On the other hand, let $$\omega_0:=i\alpha\wedge\bar\alpha + i\beta\wedge\bar\beta + i\gamma\wedge\bar\gamma$$

\noindent be the Hermitian (even balanced) metric on $X$ {\bf canonically} induced by the complex parallelisable structure of $X$. It can be easily seen that the vector space of {\bf small essential deformations} coincides with the space $H^{2,\,1}_{[\gamma]}(X)$ of [Pop18]: $$E_1^{2,\,1}(X)_{ess}= J^{2,\,1}_{\omega_0}(E_2^{2,\,1}(X)) = H^{2,\,1}_{[\gamma]}(X)\subset E_1^{2,\,1}(X).$$

\end{Ex}

\begin{Ex}\label{Ex:I_5_ess}(The {\bf manifold $I^{(5)}$}) Let $X=I^{(5)}$ be the complex parallelisable nilmanifold of complex dimension $5$ described in Example \ref{Ex:Kuranishi_I5} (i.e. the $5$-dimensional analogue of the Iwasawa manifold.) It is a page-$1$-$\partial\bar\partial$-manifold by [PSU20a, Thm. 4.7]. We will use the standard notation $\varphi_{i_1\dots i_p\bar{j}_1\dots\bar{j}_q}:= \varphi_{i_1}\wedge\dots\wedge\varphi_{i_p}\wedge\overline\varphi_{j_1}\wedge\dots\wedge\overline\varphi_{j_q}$.

\vspace{1ex}

 For every $l\in\{3,4,5\}$, the linear map $$T_l:H^{0,\,1}(X,\,T^{1,\,0}X)\longrightarrow H^{0,\,1}(X), \hspace{3ex} [\theta]\mapsto[\theta\lrcorner\varphi_l],$$

\noindent is well defined. If we set $$H^{0,\,1}_{ess}(X,\,T^{1,\,0}X):=\ker T_3\cap\ker T_4\cap\ker T_5\subset H^{0,\,1}(X,\,T^{1,\,0}X),$$

\noindent and define $H^{4,\,1}_{ess}(X)\subset H^{4,\,1}(X)$ to be the image of $H^{0,\,1}_{ess}(X,\,T^{1,\,0}X)$ under the Calabi-Yau isomorphism $H^{0,\,1}(X,\,T^{1,\,0}X)\longrightarrow H^{4,\,1}(X)$ w.r.t. $u=\varphi_1\wedge\ldots\wedge\varphi_5$, we get the following description:

$$H^{4,\,1}_{ess}(X) = \bigg\langle[\varphi_{2345\bar{1}}]_{\bar\partial},\,[\varphi_{1345\bar{1}}]_{\bar\partial},\,[\varphi_{2345\bar{2}}]_{\bar\partial},\,[\varphi_{1345\bar{2}}]_{\bar\partial}\bigg\rangle.$$

 Moreover, we have the following identities of $\C$-vector spaces: $$H^{4,\,1}_{ess}(X) = E_1^{4,\,1}(X)_{ess}:=J_{\omega_0}^{4,\,1}(E_2^{4,\,1}(X))\subset E_1^{4,\,1}(X),$$

 \noindent where $$\omega_0:=\sum\limits_{j=1}^5 i\varphi_j\wedge\overline\varphi_j,$$

 \noindent is the canonical metric of $I^{(5)}$.

\end{Ex}







\section{Deformation unobstructedness for page-$1$-$\partial\bar\partial$-manifolds}\label{section:page-1-unobstructedness}

In this section, we prove Theorem \ref{The:Introd_page-1-ddbar_unobstructedness}.

\begin{Def}\label{Def:ess_unobstructedness} Let $X$ be a {\bf Calabi-Yau page-$1$-$\partial\bar\partial$-manifold} with $\mbox{dim}_\C X=n$. Fix a non-vanishing holomorphic $(n,\,0)$-form $u$ on $X$.

  We say that the {\bf essential Kuranishi family} of $X$ is {\bf unobstructed} if every $E_2$-class in $E^{n-1,\,1}_2(X)$ admits a representative $\psi_1(t)\lrcorner u$ such that the integrability condition (\ref{eqn:integrability-cond}) is satisfied (i.e. all the equations (Eq. $(\nu)$) of $\S.$\ref{subsection:background-def} are solvable) when starting off with $\psi_1(t)\in C^\infty_{0,\,1}(X,\,T^{1,\,0}X)$.

\end{Def}

Before proving Theorem \ref{The:Introd_page-1-ddbar_unobstructedness}, we make a few comments. First, we notice an equivalent formulation for the assumption made in (ii) of Theorem \ref{The:Introd_page-1-ddbar_unobstructedness}. Obviously, the inclusion ${\cal Z}_2^{n-1,\,1}\subset{\cal Z}_1^{n-1,\,1}$ always holds.

\begin{Lem}\label{Lem:Z_equality_rewording} Let $X$ be a compact complex {\bf page-$1$-$\partial\bar\partial$-manifold} with $\mbox{dim}_\C X=n$. Then, ${\cal Z}_1^{n-1,\,1} = {\cal Z}_2^{n-1,\,1}$ if and only if every Dolbeault cohomology class of bidegree $(n-1,\,1)$ can be represented by a $d$-closed form.

\end{Lem}

\noindent {\it Proof.} Let $\alpha\in C^\infty_{n-1,\,1}(X)$ be an arbitrary $\bar\partial$-closed form, i.e. $\alpha\in{\cal Z}_1^{n-1,\,1}$. The class $\{\alpha\}_{\bar\partial}$ can be represented by a $d$-closed form if and only if there exists $\beta$ of bidegree $(n-1,\,0)$ such that $\partial(\alpha+\bar\partial\beta)=0$. This is equivalent to $\partial\alpha$ being $\partial\bar\partial$-exact, which implies that $\partial\alpha$ is $\bar\partial$-exact.

Conversely, since $X$ is a page-$1$-$\partial\bar\partial$-manifold, the $\bar\partial$-exactness of $\partial\alpha$ implies its $\partial\bar\partial$-exactness. Indeed, $\bar\partial\alpha=0$ and if $\partial\alpha$ is $\bar\partial$-exact, then $\alpha\in{\cal Z}_2^{n-1,\,1}$, so $\partial\alpha\in\partial({\cal Z}_2^{n-1,\,1})$. Now, $\partial({\cal Z}_2^{n-1,\,1}) = \mbox{Im}\,(\partial\bar\partial)$ thanks to property (i) in characterisation (F) of the page-$1$-$\partial\bar\partial$-property given in [PSU20b, Thm. 4.3] (with $r=2$). Therefore, $\partial\alpha\in\mbox{Im}\,(\partial\bar\partial)$ whenever $\alpha\in{\cal Z}_2^{n-1,\,1}$.

Summing up, the class $\{\alpha\}_{\bar\partial}$ can be represented by a $d$-closed form if and only if $\partial\alpha$ is $\bar\partial$-exact if and only if $\alpha\in{\cal Z}_2^{n-1,\,1}$.  \hfill $\Box$

\vspace{2ex}

Second, we notice that both the Iwasawa manifold $I^{(3)}$ and the $5$-dimensional Iwasawa manifold $I^{(5)}$ satisfy all the hypotheses of Theorem \ref{The:Introd_page-1-ddbar_unobstructedness}. Indeed, $I^{(3)}$ and $I^{(5)}$ are complex parallelisable nilmanifolds, so they are page-$1$-$\partial\bar\partial$-manifolds by Theorem [PSU20a, Thm. 4.7]. In particular they are also Calabi-Yau manifolds (actually, all nilmanifolds are). Moreover, we have

\begin{Lem}\label{Lem:Iwasawas_def-hypotheses} Let $X$ be either $I^{(3)}$ or $I^{(5)}$ and let $n=\mbox{dim}_\C X\in\{3,\,5\}$. Let $u:=\varphi_1\wedge\varphi_2\wedge\varphi_3 = \alpha\wedge\beta\wedge\gamma\in C^\infty_{3,\,0}(I^{(3)})$ or  $u:=\varphi_1\wedge\dots\wedge\varphi_5\in C^\infty_{5,\,0}(I^{(5)})$ according to whether $X=I^{(3)}$ or $X=I^{(5)}$, a non-vanishing holomorphic $(n,\,0)$-form on $X$.

  Then, for all $\psi_1(t), \rho_1(s)\in C^\infty_{0,\,1}(X,\,T^{1,\,0}X)$ such that $\psi_1(t)\lrcorner u, \rho_1(s)\lrcorner u\in\ker d\cup\mbox{Im}\,\partial$, we have

\begin{equation*}\label{eqn:extra-hyp_unobstructedness_example}\psi_1(t)\lrcorner(\rho_1(s)\lrcorner u)\in{\cal Z}_2^{n-2,\,2}.\end{equation*}

\end{Lem}

\noindent {\it Proof.} It is given in section \ref{section:explicit-computations}.  \hfill $\Box$

\vspace{3ex}

Finally, let us mention that both manifolds $X=I^{(3)}$ and $X=I^{(5)}$ have the property that every Dolbeault cohomology class of type $(n-1,\,1)$ can be represented by a $d$-closed form. Indeed, as seen in the proof of Lemma \ref{Lem:Iwasawas_def-hypotheses} spelt out in $\S.$\ref{section:explicit-computations}, $H^{n-1,\,1}_{\bar\partial}(X)$ is generated by the classes represented by the $\widehat\varphi_i\wedge\overline\varphi_1$'s and the $\widehat\varphi_i\wedge\overline\varphi_2$'s with $i\in\{1,2,3\}$ (in the case of $X=I^{(3)}$) and $i\in\{1,\dots ,5\}$ (in the case of $X=I^{(5)}$), where $\widehat\varphi_i$ stands for $u=\varphi_1\wedge\ldots\wedge\varphi_5$ with $\varphi_i$ omitted. All the forms $\widehat\varphi_i\wedge\overline\varphi_\lambda$, with $\lambda\in\{1,2\}$, are $d$-closed.

Note that the hypotheses of Theorem \ref{The:Introd_page-1-ddbar_unobstructedness}, all of which are satisfied by $X=I^{(3)}$ and $X=I^{(5)}$, have the advantage of being cohomological in nature, hence fairly general and not restricted to the class of nilmanifolds. Indeed, there is no mention of any structure equations in Theorem \ref{The:Introd_page-1-ddbar_unobstructedness}.

\vspace{2ex}

Finally, for the reader's convenience, we recall the following classical result (cf. Lemma 3.1. in [Tia87], Lemma 1.2.4. in [Tod89]) that will be made a key use of in the proof of Theorem \ref{The:Introd_page-1-ddbar_unobstructedness}.

\begin{Lem}(Tian-Todorov Lemma)\label{Lem:Tian-Todorov} Let $X$ be a compact complex manifold ($n=\mbox{dim}_{\C}X$) such that $K_X$ is {\bf trivial}. Then, for any forms $\theta_1, \theta_2\in C^{\infty}_{0,\, 1}(X,\, T^{1,\, 0}X)$ such that $\partial(\theta_1\lrcorner u)=\partial(\theta_2\lrcorner u)=0$, we have

$$[\theta_1,\, \theta_2]\lrcorner u\in\mbox{Im}\,\partial.$$

\noindent More precisely, the identity \begin{equation}\label{eqn:basic-trick}[\theta_1,\, \theta_2]\lrcorner u = -\partial(\theta_1\lrcorner(\theta_2\lrcorner u))\end{equation} holds for $\theta_1, \theta_2\in C^{\infty}_{0,\, 1}(X,\, T^{1,\, 0}X)$ whenever $\partial(\theta_1\lrcorner u) = \partial(\theta_2\lrcorner u)=0$.

\end{Lem}

\vspace{3ex}

\begin{proof}[Proof of Theorem \ref{The:Introd_page-1-ddbar_unobstructedness}.] Let $\{\eta_1\}_{E_2}\in E_2^{n-1,\,1}(X)$ be an arbitrary nonzero class. Pick any $d$-closed representative $\eta_1\in C^\infty_{n-1,\,1}(X)$ of it. A $d$-closed representative exists thanks to the page-$1$-$\partial\bar\partial$-assumption on $X$. Under the extra assumption ${\cal Z}_1^{n-1,\,1} = {\cal Z}_2^{n-1,\,1}$ of (ii), there is even a $d$-closed representative $\eta_1$ in every Dolbeault class $\{\eta_1\}_{E_1}\in E_1^{n-1,\,1}(X)$, thanks to Lemma \ref{Lem:Z_equality_rewording}. So, we choose an arbitrary $d$-closed form $\eta_1\in C^\infty_{n-1,\,1}(X)$ that represents an arbitrary nonzero class in either $E_2^{n-1,\,1}(X)$ or $E_1^{n-1,\,1}(X)$ depending on whether we are in case (i) or in case (ii). By the Calabi-Yau isomorphism (\ref{eqn:C-Y_isom_forms}), there exists a unique $\psi_1\in C^\infty_{0,\,1}(X,\,T^{1,\,0}X)$ such that $$\psi_1\lrcorner u=\eta_1.$$

 We will prove the existence of forms $\psi_\nu\in C^\infty_{0,\,1}(X,\,T^{1,\,0}X)$, with $\nu\in\N^\star$ and $\psi_1$ being the already fixed such form, that satisfy the equations \begin{equation*}\label{eqn:integrability-condition_nu_bis}\bar\partial\psi_\nu = \frac{1}{2}\,\sum\limits_{\mu=1}^{\nu-1}[\psi_\mu,\,\psi_{\nu-\mu}] \hspace{3ex} (\mbox{Eq.}\,\,(\nu-1)), \hspace{3ex} \nu\geq 2, \end{equation*}

\noindent which, as recalled in $\S.$\ref{subsection:background-def}, are equivalent to the integrability condition $\bar\partial\psi(\tau) = (1/2)\,[\psi(\tau),\,\psi(\tau)]$ being satisfied by the form $\psi(\tau):=\psi_1\,\tau + \psi_2\,\tau^2+\dots + \psi_N\,\tau^N + \dots \in C^\infty_{0,\,1}(X,\,T^{1,\,0}X)$ for all $\tau\in\C$ with $|\tau|$ sufficiently small. The convergence in a H\"older norm of the series defining $\psi(\tau)$ for $|\tau|$ small enough is guaranteed by the general Kuranishi theory (cf. [Kur62]), while the resulting $\psi(\tau)$ defines a complex structure $\bar\partial_\tau$ on $X$ that identifies on functions with $\bar\partial - \psi(\tau)$ and represents the infinitesimal deformation of the original complex structure $\bar\partial$ of $X$ in the direction of $[\psi_1]\in H^{0,\,1}(X,\,T^{1,\,0}X)$.

Since $\partial(\psi_1\lrcorner u) = \partial\eta_1 =0$, the Tian-Todorov Lemma \ref{Lem:Tian-Todorov} guarantees that $[\psi_1,\,\psi_1]\lrcorner u\in\mbox{Im}\,\partial$ and $$[\psi_1,\,\psi_1]\lrcorner u = -\partial(\psi_1\lrcorner(\psi_1\lrcorner u)).$$

\noindent On the other hand, $\bar\partial\eta_1=0$, hence $\bar\partial\psi_1=0$, hence $\psi_1\lrcorner(\psi_1\lrcorner u)\in\ker\bar\partial$. We even have the stronger property $\psi_1\lrcorner(\psi_1\lrcorner u)\in{\cal Z}_2^{n-2,\,2}$ thanks to assumption (\ref{eqn:extra-hyp_unobstructedness}), since $\psi_1\lrcorner u\in\ker d$. Therefore, $$[\psi_1,\,\psi_1]\lrcorner u = -\partial(\psi_1\lrcorner(\psi_1\lrcorner u))\in\partial({\cal Z}_2^{n-2,\,2}) = \mbox{Im}\,(\partial\bar\partial),$$

\noindent the last identity being a consequence of the page-$1$-$\partial\bar\partial$-assumption on $X$. (See (i) of property (F) in Theorem \ref{The:page_r-1_ddbar-equivalence_bis}.)

Thus, there exists a form $\Phi_2\in C^\infty_{n-2,\,1}(X)$ such that $$\bar\partial\partial\Phi_2 = \frac{1}{2}\,[\psi_1,\,\psi_1]\lrcorner u.$$

\noindent If we fix an arbitrary Hermitian metric $\omega$ on $X$, we choose $\Phi_2$ as the unique solution of the above equation with the extra property $\Phi_2\in\mbox{Im}\,(\partial\bar\partial)^\star$. This is the minimal $L^2_\omega$-norm solution, as follows from the $3$-space orthogonal decomposition of $C^\infty_{n-2,\,1}(X)$ induced by the Aeppli Laplacian (see [Sch07]). Let $\eta_2:=\partial\Phi_2\in C^\infty_{n-1,\,1}(X)$. Thanks to the Calabi-Yau isomorphism (\ref{eqn:C-Y_isom_forms}), there exists a unique $\psi_2\in C^\infty_{0,\,1}(X,\,T^{1,\,0}X)$ such that $\psi_2\lrcorner u =\eta_2$. In particular, $\partial(\psi_2\lrcorner u)=0$ and $(\bar\partial\psi_2) \lrcorner u = \bar\partial(\psi_2\lrcorner u) =\bar\partial\eta_2 = (1/2)\,[\psi_1,\,\psi_1]\lrcorner u$. This means that $$\bar\partial\psi_2 = \frac{1}{2}\,[\psi_1,\,\psi_1],$$

\noindent so $\psi_2$ is a solution of (Eq.\,$1$). Moreover, by construction, $\psi_2$ has the extra key property that $\psi_2\lrcorner u\in\mbox{Im}\,\partial$.

We continue inductively to construct the forms $(\psi_N)_{N\geq 3}$. Suppose the forms $\psi_1, \psi_2, \dots , \psi_{N-1}\in C^\infty_{0,\,1}(X,\,T^{1,\,0}X)$ have been constructed as solutions of the equations (Eq.\,$(\nu-1)$) for all $\nu\in\{2,\dots , N-1\}$ with the further property $\psi_2\lrcorner u,\dots ,  \psi_{N-1}\lrcorner u\in\mbox{Im}\,\partial$. (Recall that $\psi_1\lrcorner u\in\ker d$.) Since $\partial(\psi_1\lrcorner u) = \partial(\psi_2\lrcorner u) = \dots = \partial(\psi_{N-1}\lrcorner u) =0$, the Tian-Todorov Lemma \ref{Lem:Tian-Todorov} guarantees that $[\psi_\mu,\,\psi_{N-\mu}]\lrcorner u\in\mbox{Im}\,\partial$ for all $\mu\in\{1,\dots , N-1\}$ and yields the first identity below: $$\sum\limits_{\mu=1}^{N-1}[\psi_\mu,\,\psi_{N-\mu}]\lrcorner u = -\partial\bigg(\sum\limits_{\mu=1}^{N-1}\psi_\mu\lrcorner(\psi_{N-\mu}\lrcorner u)\bigg)\in\partial({\cal Z}_2^{n-2,\,2}) = \mbox{Im}\,(\partial\bar\partial),$$

\noindent where the relation ``$\in$'' follows from assumption (\ref{eqn:extra-hyp_unobstructedness}) and the last identity is a consequence of the page-$1$-$\partial\bar\partial$-assumption on $X$. (See (i) of property (F) in Theorem \ref{The:page_r-1_ddbar-equivalence_bis}.)

Thus, there exists a form $\Phi_N\in C^\infty_{n-2,\,1}(X)$ such that $$\bar\partial\partial\Phi_N = \frac{1}{2}\,\sum\limits_{\mu=1}^{N-1}[\psi_\mu,\,\psi_{N-\mu}]\lrcorner u.$$

\noindent We choose $\Phi_N$ to be the solution of minimal $L^2_\omega$-norm of the above equation, so $\Phi_N\in\mbox{Im}\,(\partial\bar\partial)^\star$. Let $\eta_N:=\partial\Phi_N\in C^\infty_{n-1,\,1}(X)$. Thanks to the Calabi-Yau isomorphism (\ref{eqn:C-Y_isom_forms}), there exists a unique $\psi_N\in C^\infty_{0,\,1}(X,\,T^{1,\,0}X)$ such that $\psi_N\lrcorner u = \eta_N$. Hence, $(\bar\partial\psi_N)\lrcorner u =\bar\partial(\psi_N\lrcorner u) = \bar\partial\eta_N = \bar\partial\partial\Phi_N$, so $$\bar\partial\psi_N = \frac{1}{2}\,\sum\limits_{\mu=1}^{N-1}[\psi_\mu,\,\psi_{N-\mu}],$$

\noindent which means that $\psi_N$ is a solution of (Eq.\,$(N-1)$). Moreover, by construction, $\psi_N$ has the extra key property that $\psi_N\lrcorner u\in\mbox{Im}\,\partial$.

This finishes the induction process and completes the proof of Theorem \ref{The:Introd_page-1-ddbar_unobstructedness}.\end{proof}

The proof of Theorem \ref{The:Introd_page-1-ddbar_unobstructedness} shows that it suffices to check condition (\ref{eqn:extra-hyp_unobstructedness}) on a small subset of all forms. For instance, we have

\begin{Rem}\label{Rem:sub-double-complex} Suppose there is a sub-double-complex $C\subset A_X:=(C^{\infty}_{\Cdot,\Cdot}(X,\C),\partial,\bar\partial)$ such that

  \vspace{1ex}

  $(1)$\, the inclusion $C\subset A_X$ is an $E_1$-isomorphism;\footnote{Recall [Ste21] that this means it induces an isomorphism both in Dolbeault and conjugate Dolbeault cohomology (the latter being automatic if $C$ is a real sub-complex).}

\vspace{1ex}

$(2)$\, writing $T=\{\omega\in C_{0,1}^\infty(X,TX^{1,0})\mid \omega\lrcorner u\in (\ker d\cup \mbox{Im} \partial)\cap C\}$, the set $\partial(T\lrcorner T\lrcorner u)=[T,T]\lrcorner u$ is contained in $C$.

\vspace{1ex}

Then, for the conclusion of Theorem \ref{The:Introd_page-1-ddbar_unobstructedness} to hold, it suffices to check that the forms in $T$ satisfy condition \eqref{eqn:extra-hyp_unobstructedness} in Theorem \ref{The:Introd_page-1-ddbar_unobstructedness}.

\end{Rem}

\noindent {\it Proof.} By property $(1)$, we have $E_2^{n-1,1}(X)=E_2^{n-1,1}(C)$. We may therefore start the proof of Theorem \ref{The:Introd_page-1-ddbar_unobstructedness} by picking $\eta_1\in C$. By property $(2)$, $\psi_1\in T$. Again by property $(1)$, the inclusion $C\to A_X$ induces isomorphisms in Bott-Chern cohomology and higher pages of the Fr\"olicher spectral sequence of $C$ and $A_X$ [Ste21, Cor. 13], whenever a form in $C$ is exact in any way (w.r.t. $\partial$, $\bar\partial$, $\partial\bar\partial$, $d_r$,...) in $A_X$, one is able to find a primitive in $C$. Using this and property $(2)$, whenever we use the Calabi-Yau isomorphism, we see that, at each step in the proof of Theorem \ref{The:Introd_page-1-ddbar_unobstructedness}, one may take the solutions to (Eq.\,$(\nu)$) to lie in $C$.

 \hfill $\Box$

\section{Examples, applications and explicit computations}\label{section:explicit-computations}

In this section, we apply our results to certain classes of compact {\it complex paralellisable solvmanifolds} of complex dimension $3$ that were studied by Nakamura in [Nak75]. For the reader's convenience, we start by giving a brief rundown of the background by following Hasegawa's more recent treatment of these manifolds in [Has10], where Nakamura's discussion was expanded.

Let $X$ be a compact {\it complex paralellisable solvmanifold} with $\mbox{dim}_\C X=n$. It is standard that any such $X$ arises as a quotient $X=G/\Gamma$, where $G$ is a {\it simply connected solvable complex} Lie group and $\Gamma\subset G$ is a {\it co-compact lattice} (i.e. a discrete subgroup). Any such $G$ is {\it unimodular} (i.e. the left-invariant Haar measure of $G$ is also right-invariant, a fact that is equivalent to $|\det Ad_g|=1$ for every $g\in G$). This is equivalent to the Lie algebra $\fg$ of $G$ being {\it unimodular} (i.e. $\mbox{tr}\, (ad_\xi)=0$ for every $\xi\in\fg$).

Now, let $n=3$. Fix a $\C$-basis $\{X, Y, Z\}$ of $\fg$. The {\it unimodular solvable complex} Lie algebras $\fg$ of complex dimension $3$ are classified into the following types ([Nak75], [Has10]):

\vspace{1ex}

$(1)$\, $\fg$ is {\it Abelian}: $[X,\,Y] = [Y,\,Z] = [Z,\,X] = 0$;

\vspace{1ex}

$(2)$\, $\fg$ is {\it nilpotent}: $[X,\,Y] = Z,$ \hspace{2ex} $[X,\,Z] = [Y,\,Z] = 0$;

\vspace{1ex}

Note that in this case we have: $\fg^1:=[\fg,\,\fg]=\langle Z\rangle$, so $\fg^2:=[[\fg,\,\fg],\, \fg] = [\langle Z\rangle,\, \langle X,\,Y,\,Z\rangle] = 0$, hence $\fg$ is {\it $2$-step nilpotent}.

\vspace{1ex}

$(3)$\, $\fg$ is {\it non-nilpotent} (but {\it solvable}): $[X,\,Y] = -Y,$ \hspace{2ex} $[X,\,Z] = Z,$ \hspace{2ex} $[Y,\,Z] = 0$.

\vspace{1ex}

Note that in this case we have: $\fg^1:=[\fg,\,\fg]= \langle Y,\,Z\rangle$, so $\fg^2:=[[\fg,\,\fg],\, \fg] = \langle Y,\,Z\rangle = \fg^1$, so $\fg$ is indeed non-nilpotent. However, $\fg^{(2)}:=[[\fg,\,\fg],\,[\fg,\,\fg]] = [\langle Y,\,Z\rangle,\,\langle Y,\,Z\rangle] = 0$, so $\fg$ is indeed solvable.

The solvmanifolds $X=G/\Gamma$ corresponding to Lie groups $G$ whose Lie algebras $\fg$ are of this type $(3)$ are usually called {\bf Nakamura manifolds}. They are not nilmanifolds.

\vspace{2ex}

The {\it lattices} $\Gamma\subset G$ of the {\it simply connected solvable complex} Lie groups $G$ whose Lie algebras $\fg$ belong to the respective above classes are completely determined as follows. (See ([Nak75], [Has10]).)

\vspace{1ex}

$(1)$\, If $\fg$ is {\it Abelian}, then $G=(\C^3,\,+)$ and any lattice $\Gamma\subset G$ is $\Z$-generated by an $\R$-basis of $\C^3\simeq\R^6$. Hence, $X:=G/\Gamma$ is a {\it complex torus}.

\vspace{1ex}

$(2)$\, If $\fg$ is {\it nilpotent}, then $G$ is the {\it semi-direct product} $G=\C^2\rtimes_\phi\C$ with $$\phi:\C\longrightarrow\mbox{Aut}(\C^2), \hspace{2ex} \phi(x) = \begin{pmatrix} 1 & 0 \\

  x & 1

\end{pmatrix}.$$ This means that the {\it group operation} on $G=\C^2\rtimes_\phi\C$ is defined by $$(a,\,b,\,c)\star(x,\,y,\,z):= (a+x,\, (b,\,c) + \phi(a)(y,\,z)) = (a+x,\,b+y,\,c+z+ay).$$

Any lattice $\Gamma\subset G$ is of the shape $$\Gamma = \Delta\rtimes_\phi\Lambda, \hspace{3ex} \mbox{where} \hspace{2ex} \Delta\subset\C^2\simeq\R^4 \hspace{1ex} \mbox{and} \hspace{1ex} \Lambda\subset\C\simeq\R^2 \hspace{2ex} \mbox{are lattices}.$$

Hasegawa [Has10] goes on to show that we can assume, without loss of generality, these lattices to be generated over $\Z$ as follows: $$\Lambda = \langle 1,\,\lambda\rangle, \hspace{2ex} \Delta = \langle (\alpha_1,\,\beta_1),\, (\alpha_2,\,\beta_2),\, (0,\,\alpha_1),\,(0,\,\alpha_2)\rangle,$$ where $\lambda\in\C\setminus\R$, $\beta_1,\beta_2\in\C$ are arbitrary and $\alpha_1,\alpha_2\in\C$ are $\R$-linearly independent such that $(\alpha_1,\,\alpha_2)$ is a $\lambda$-eigenvector of some $A\in GL(2,\,\Z)$.

The Iwasawa manifold $I^{(3)} = G/\Gamma$ belongs to this class. It is obtained for the choice of lattice $\Gamma$ defined by $\lambda=i$, $\beta_1=\beta_2=0$, $\alpha_1=1$, $\alpha_2=i$.

\vspace{1ex}

$(3)$\,  If $\fg$ is {\it non-nilpotent}, then $G$ is the {\it semi-direct product} $G=\C^2\rtimes_\phi\C$ with $$\phi:\C\longrightarrow\mbox{Aut}(\C^2), \hspace{2ex} \phi(x) = \begin{pmatrix} e^{x} & 0 \\

  0 & e^{-x}

\end{pmatrix}.$$ This means that the {\it group operation} on $G=\C^2\rtimes_\phi\C$ is defined by $$(a,\,b,\,c)\star(x,\,y,\,z):= (a+x,\, (b,\,c) + \phi(a)(y,\,z)) = (a+x,\,b+e^{a}y,\,c+e^{-a}z).$$

Any lattice $\Gamma\subset G$ is of the shape $$\Gamma = \Delta\rtimes_\phi\Lambda, \hspace{3ex} \mbox{where} \hspace{2ex} \Delta\subset\C^2\simeq\R^4 \hspace{1ex} \mbox{and} \hspace{1ex} \Lambda\subset\C\simeq\R^2 \hspace{2ex} \mbox{are lattices}.$$

Hasegawa [Has10] goes on to show these lattices to be generated over $\Z$ as follows: $$\Lambda = \langle\lambda,\,\mu\rangle, \hspace{2ex} \Delta = \langle (\alpha_1,\,\beta_1),\, (\alpha_2,\,\beta_2),\, (\alpha_3,\,\beta_3),\,(\alpha_4,\,\beta_4)\rangle,$$ where $\lambda,\,\mu\in\C$ satisfy the following condition. There exist {\it commuting semi-simple} matrices $A,B\in SL_4(\Z)$ and vectors $\alpha=(\alpha_1,\alpha_2,\alpha_3,\alpha_4),\, \beta=(\beta_1,\beta_2,\beta_3,\beta_4)\in\C^4$ such that $\alpha$ and $\beta$ are eigenvectors for $A$ with eigenvalues $e^{-\lambda}$, resp. $e^\lambda$, while $\alpha$ and $\beta$ are also eigenvectors for $B$ with eigenvalues $e^{-\mu}$, resp. $e^\mu$. Conversely, all lattices $\Gamma$ of $G=\C^2\rtimes_\phi\C$ arise in this way from arbitrary {\it commuting semi-simple} matrices $A,B\in SL_4(\Z)$.

This third class of solvmanifolds (corresponding to the case where $\fg$ is {\it non-nilpotent}) has two subclasses:

\vspace{1ex}

$\bullet$ {\it subclass $(3a)$}: when either $Im(\lambda)\notin\pi\Z$ or $Im(\mu)\notin\pi\Z$.

In this case, $h^{0,\,1}_{\bar\partial}(X) =1$, $\mbox{dim}_\C H^{0,\,1}_{\bar\partial}(X,\,T^{1,\,0}X) = 3$ and $X$ has {\it unobstructed} deformations [Nak75, p. 99].

\vspace{1ex}

$\bullet$ {\it subclass $(3b)$}: when $Im(\lambda),Im(\mu)\in\pi\Z$.

In this case, $h^{0,\,1}_{\bar\partial}(X) =3$, $\mbox{dim}_\C H^{0,\,1}_{\bar\partial}(X,\,T^{1,\,0}X) = 9$ and $X$ has {\it obstructed} deformations [Nak75, eqn. (3.3.)].

\subsection{Case of the Nakamura manifolds of type $(3b)$}\label{subsection:Nakamura}

We now prove Proposition \ref{Prop:obstructed-unobstructed_ess-non-ess} by showing that the Nakamura manifolds of type $(3b)$ constitute an example of manifolds satisfying the conditions of that statement. However, we cannot immediately apply Theorem \ref{The:Introd_page-1-ddbar_unobstructedness} because of

\begin{Obs}\label{Obs:no-hypothesis_yes-result} The Nakamura manifolds of class $(3b)$ do not satisfy hypothesis (\ref{eqn:extra-hyp_unobstructedness}) of Theorem \ref{The:Introd_page-1-ddbar_unobstructedness}.

\end{Obs}

\noindent {\it Proof.} Let us denote the coordinates on $G=\C^2\rtimes \C$ by $(z_3,z_2,z_1)$, which is the ordering adopted in [Has10] and [AK17]. It differs from the original ordering in [Nak75], where the roles of $z_3$ and $z_2$ are reversed. We will use standard abbreviations like $dz_{12\bar{3}}:=dz_1\wedge dz_2\wedge d\bar{z}_3$ and set $u:=dz_{123}$. Set
\[\psi:=e^{-2z_1}d\bar{z}_2\otimes\frac{\partial}{\partial z_3}\text{ and }\rho:=d\bar{z}_1\otimes \frac{\partial}{\partial z_1}.
\]
Then
\[
\del(\psi\lrcorner(\rho\lrcorner u))=-2e^{-2z_1}dz_{12\overline{12}},
\]
which defines a nonzero class in $H_{BC}^{2,2}$ by [AK17]. Therefore it is not $\bar\partial\partial$-exact, which implies that $(\psi\lrcorner(\rho\lrcorner u))\not\in Z_2^{n-2,2}$ since $\del Z_2^{n-2,2}=\operatorname{Im} (\del\delbar)$ by the page-1-$\partial\bar\partial$ hypothesis.\hfill $\Box$

\vspace{2ex}
We now show that the Nakamura manifolds of class $(3b)$ have unobstructed essential deformations despite them not satisfying hypothesis (\ref{eqn:extra-hyp_unobstructedness}) of Theorem \ref{The:Introd_page-1-ddbar_unobstructedness}. Consider the left-invariant forms on $G$ given by $\eta_1:=dz_1$, $\eta_2:=e^{-z_1}dz_2$, $\eta_3:=e^{z_1}dz_3$ and equip the Nakamura manifold $X$ with the metric $\omega_0:=\sum_{i=1}^{3}\eta_i\otimes\bar{\eta}_i$.
\vspace{2ex}

\noindent {\it Proof of Proposition \ref{Prop:obstructed-unobstructed_ess-non-ess}.} From table [AK17, table 7], we get, with the same notational conventions as in the previous proof,
\[
E_2^{2,1}(X)=\langle [dz_{12\bar{3}}]_{E_2}, [dz_{13\bar{2}}]_{E_2}, [dz_{23\bar{1}}]_{E_2}\rangle\overset{J_{\omega_0}^{2,1}}\cong\langle [dz_{12\bar{3}}]_{E_1}, [dz_{13\bar{2}}]_{E_1}, [dz_{23\bar{1}}]_{E_1}\rangle= E_1^{p,q}(X)_{ess}^{\omega_0},
\]
where we use that the chosen representatives are closed and co-closed. Under the (inverse of) the Calabi-Yau isomorphism this space corresponds to:
\[
H_{ess}^{0,1}(X,TX^{1,0})=\langle d\bar{z}_3\otimes \frac{\partial}{\partial z_3}, d\bar{z}_2\otimes \frac{\partial}{\partial z_2}, d\bar{z}_1\otimes \frac{\partial}{\partial z_1}\rangle
\]
Since in [Nak75] the roles of $z_2$ and $z_3$ are reversed, in the notation of [Nak75] we have:
\[
H_{ess}^{0,1}(X,TX^{1,0})=\langle \theta_2\varphi^\ast_2, \theta_3 \varphi_3^\ast, \theta_1\varphi_1^\ast\rangle.
\]
This means that an essential deformation corresponds to an element
\[
\mu=t_{22}\cdot \theta_2\varphi^\ast_2 + t_{33}\cdot \theta_3 \varphi_3^\ast + t_{11}\cdot \theta_1\varphi_1^\ast.
\]
In particular, $t_{i\lambda}=0$ for $(i,\lambda)\neq (2,2), (3,3), (1,1)$. So the conditions in [Nak75, eqn (3.3.)] are trivially satisfied, i.e. the essential deformations are unobstructed.

Note also that, since $t_{12}=0$, cases $(3)$ and $(4)$ of [Nak75, p.98f] do not occur. This implies that, e.g., the universal covering space of any small essential deformation $X_t$ of $X_0=X$ is always $\C^3$.  \hfill $\Box$

\subsection{Further computations}\label{subsection:further-computations}

In this subsection, we spell out the proof of Lemma \ref{Lem:Iwasawas_def-hypotheses}. The subcomplex of left-invariant forms on $I^{(3)}$ and $I^{(5)}$ satisfies the conditions of Remark \ref{Rem:sub-double-complex} by [Sak76], which is why in the following we will, without further mentioning, work with left-invariant forms only. \\

 $\bullet$\, {\it Case where $X=I^{(3)}$.} We use the notation of Example \ref{Ex:Iwasawa_dim3_ess}, but also put $\varphi_1:=\alpha$, $\varphi_2:=\beta$ and $\varphi_3:=\gamma$. We have: $d\varphi_1=d\varphi_2=0$ and $d\varphi_3=-\varphi_1\wedge\varphi_2$. The dual basis of $(1,\,0)$-vector fields consists of $$\theta_1=\frac{\partial}{\partial z_1}, \hspace{2ex} \theta_2=\frac{\partial}{\partial z_2} + z_1\,\frac{\partial}{\partial z_3}, \hspace{2ex} \theta_3=\frac{\partial}{\partial z_3},$$

\noindent (actually of the vector fields induced by these ones on $X$ by passage to the quotient) whose mutual Lie brackets are as follows: $$[\theta_1,\,\theta_2] = -[\theta_2,\,\theta_1]=\theta_3 \hspace{3ex}\mbox{and}\hspace{3ex} [\theta_i,\,\theta_j]=0 \hspace{2ex}\mbox{whenever}\hspace{2ex} \{i,\,j\}\neq\{1,\,2\}.$$

\noindent In particular, $H^{0,\,1}(X,\,T^{1.\,0}X) = \langle[\overline\varphi_1\otimes\theta_i],\,[\overline\varphi_2\otimes\theta_i]\,\mid\,i=1,\dots , 3\rangle$, so $\mbox{dim}_\C H^{0,\,1}(X,\,T^{1.\,0}X)=6$.

Note that all the $(2,\,1)$-forms $(\overline\varphi_1\otimes\theta_i)\lrcorner u$ and $(\overline\varphi_2\otimes\theta_i)\lrcorner u$ are $d$-closed for $i\in\{1,\,2,\,3\}$, so every Dolbeault class in $H^{2,\,1}_{\bar\partial}(X)$ can be represented by a $d$-closed form.

\vspace{2ex}

 (a)\, Let $\psi_1(t), \rho_1(s)\in C^\infty_{0,\,1}(X,\,T^{1,\,0}X)$ such that $\psi_1(t)\lrcorner u, \rho_1(s)\lrcorner u\in\ker d$. Then, $$\psi_1(t)=\sum\limits_{i=1}^3\sum\limits_{\lambda=1}^2 t_{i\lambda}\,\theta_i\overline\varphi_\lambda, \hspace{1ex}\mbox{so}\hspace{1ex}\psi_1(t)\lrcorner u = \sum\limits_{i=1}^3 (-1)^{i-1}\,\sum\limits_{\lambda=1}^2 t_{i\lambda}\,\overline\varphi_\lambda\wedge\widehat\varphi_i,$$ $$\rho_1(s)=\sum\limits_{j=1}^3\sum\limits_{\mu=1}^2 s_{j\mu}\,\theta_j\overline\varphi_\mu, \hspace{1ex}\mbox{so}\hspace{1ex}\rho_1(s)\lrcorner u = \sum\limits_{j=1}^3 (-1)^{j-1}\,\sum\limits_{\mu=1}^2 s_{j\mu}\,\overline\varphi_\mu\wedge\widehat\varphi_j,$$

\noindent where $\widehat\varphi_j$ stands for $\varphi_1\wedge\varphi_2\wedge\varphi_3$ with $\varphi_j$ omitted.

Since $\psi_1(t)\lrcorner u, \rho_1(s)\lrcorner u\in\ker\bar\partial$, $\psi_1(t)$ and $\rho_1(s)$ are $\bar\partial$-closed for the $\bar\partial$ of the holomorphic structure of $T^{1,\,0}X$, hence $\psi_1(t)\lrcorner(\rho_1(s)\lrcorner u)\in{\cal Z}_1^{1,\,2}$. Moreover, since $\psi_1(t)\lrcorner u, \rho_1(s)\lrcorner u\in\ker\partial$, the Tian-Todorov Lemma \ref{Lem:Tian-Todorov} ensures that $$\partial(\psi_1(t)\lrcorner(\rho_1(s)\lrcorner u)) = [\psi_1(t)\lrcorner u,\,\rho_1(s)\lrcorner u],$$

\noindent where $[\psi_1(t)\lrcorner u,\,\rho_1(s)\lrcorner u]$ is the scalar-valued $(n-1,\,2)$-form defined by the identity $[\psi_1(t)\lrcorner u,\,\rho_1(s)\lrcorner u] = [\psi_1(t),\,\rho_1(s)]\lrcorner u$. So, we have to show that $[\psi_1(t)\lrcorner u,\,\rho_1(s)\lrcorner u]$ is $\bar\partial$-exact. We get:

\begin{eqnarray*}[\psi_1(t),\,\rho_1(s)] = \sum\limits_{1\leq i,\,j\leq 3}\sum\limits_{1\leq\lambda,\,\mu\leq 2}t_{i\lambda}\,s_{j\mu}\,[\theta_i,\,\theta_j]\,\overline\varphi_\lambda\wedge\overline\varphi_\mu = D_{t,s}\,\theta_3\,\overline\varphi_1\wedge\overline\varphi_2,\end{eqnarray*}

\noindent where $D_{t,s}=(t_{11}\,s_{22} + t_{22}\,s_{11} - t_{12}\,s_{21} - t_{21}\,s_{12})$. Hence, \begin{eqnarray*}[\psi_1(t),\,\rho_1(s)]\lrcorner u = D_{t,s}\,\varphi_1\wedge\varphi_2\wedge\overline\varphi_1\wedge\overline\varphi_2 = \bar\partial(D_{t,s}\,\partial\varphi_3\wedge\overline\varphi_3) = \bar\partial\partial(D_{t,s}\,\varphi_3\wedge\overline\varphi_3)\in\mbox{Im}\,\bar\partial,\end{eqnarray*}

\noindent as desired.

We conclude that $\psi_1(t)\lrcorner(\rho_1(s)\lrcorner u)\in{\cal Z}_1^{1,\,2}$ and $\partial(\psi_1(t)\lrcorner(\rho_1(s)\lrcorner u))\in\mbox{Im}\,\bar\partial$, hence $\psi_1(t)\lrcorner(\rho_1(s)\lrcorner u)\in{\cal Z}_2^{1,\,2}$, as desired.

\vspace{2ex}

 (b)\, Let $\psi_1(t), \rho_1(s)\in C^\infty_{0,\,1}(X,\,T^{1,\,0}X)$ such that $\psi_1(t)\lrcorner u\in\ker d$ and $\rho_1(s)\lrcorner u\in\mbox{Im}\,\partial$. Then, $\psi_1(t)=\sum_{1\leq i\leq 3}\sum_{1\leq\lambda\leq 2} t_{i\lambda}\,\theta_i\overline\varphi_\lambda$ and $\rho_1(s)=(\sum_{1\leq\mu\leq 3} s_{\mu}\,\overline\varphi_\mu)\,\theta_3,$ so $$\rho_1(s)\lrcorner u = \sum_{1\leq\mu\leq 3} s_{\mu}\,\overline\varphi_\mu\wedge\varphi_1\wedge\varphi_2 = \partial(-\sum_{1\leq\mu\leq 3} s_{\mu}\,\varphi_3\wedge\overline\varphi_\mu)\in\mbox{Im}\,\partial.$$

On the one hand, we get $\displaystyle\psi_1(t)\lrcorner(\rho_1(s)\lrcorner u) = \sum\limits_{\lambda=1}^2\sum\limits_{\mu=1}^3 t_{1\lambda}\,s_\mu\,\overline\varphi_\lambda\wedge\overline\varphi_\mu\wedge\varphi_2 - \sum\limits_{\lambda=1}^2\sum\limits_{\mu=1}^3 t_{2\lambda}\,s_\mu\,\overline\varphi_\lambda\wedge\overline\varphi_\mu\wedge\varphi_1,$

\noindent hence $\displaystyle\bar\partial(\psi_1(t)\lrcorner(\rho_1(s)\lrcorner u)) = -\sum\limits_{\lambda=1}^2 t_{1\lambda}\,s_3\,\overline\varphi_\lambda\wedge\bar\partial\overline\varphi_3\wedge\varphi_2 + \sum\limits_{\lambda=1}^2 t_{2\lambda}\,s_3\,\overline\varphi_\lambda\wedge\bar\partial\overline\varphi_3\wedge\varphi_1 =0$ because $\bar\partial\overline\varphi_3 = -\overline\varphi_1\wedge\overline\varphi_2$. Thus, $\psi_1(t)\lrcorner(\rho_1(s)\lrcorner u)\in\ker\bar\partial$.

On the other hand, since $[\theta_i,\,\theta_3] =0$ for all $i$, we get $$\partial(\psi_1(t)\lrcorner(\rho_1(s)\lrcorner u)) = [\psi_1(t),\,\rho_1(s)]\lrcorner u = \sum\limits_{i=1}^3\sum\limits_{\lambda=1}^2\sum\limits_{\mu=1}^3 t_{i\lambda}\,s_\mu\,\overline\varphi_\lambda\wedge\overline\varphi_\mu\,[\theta_i,\,\theta_3] =0$$

We conclude that $\psi_1(t)\lrcorner(\rho_1(s)\lrcorner u)\in{\cal Z}_2^{1,\,2}$.

\vspace{2ex}

 (c)\, If $\psi_1(t), \rho_1(s)\in C^\infty_{0,\,1}(X,\,T^{1,\,0}X)$ are such that $\psi_1(t)\lrcorner u$ and $\rho_1(s)\lrcorner u$ both lie in $\mbox{Im}\,\partial$, then $\psi_1(t)= (\sum_{1\leq\lambda\leq 3} t_{\lambda}\,\overline\varphi_\lambda)\,\theta_3$ and $\rho_1(s)=(\sum_{1\leq\mu\leq 3} s_{\mu}\,\overline\varphi_\mu)\,\theta_3$. We get $$\psi_1(t)\lrcorner(\rho_1(s)\lrcorner u) = -(\sum_{1\leq\lambda\leq 3} t_{\lambda}\,\overline\varphi_\lambda)\wedge \sum_{1\leq\mu\leq 3} s_{\mu}\,\overline\varphi_\mu\wedge[\theta_3\lrcorner(\varphi_1\wedge\varphi_2)] = 0$$

\noindent since $\theta_3\lrcorner\varphi_1 = \theta_3\lrcorner\varphi_2 = 0$. In particular, $\psi_1(t)\lrcorner(\rho_1(s)\lrcorner u)\in{\cal Z}_2^{1,\,2}$.

\vspace{3ex}

$\bullet$\, {\it Case where $X=I^{(5)}$.} We use the notation of Example \ref{Ex:Kuranishi_I5}.

\vspace{1ex}

(a)\, Let $\psi_1(t), \rho_1(s)\in C^\infty_{0,\,1}(X,\,T^{1,\,0}X)$ such that $\psi_1(t)\lrcorner u, \rho_1(s)\lrcorner u\in\ker d$. Then, $$\psi_1(t)=\sum\limits_{i=1}^5\sum\limits_{\lambda=1}^2 t_{i\lambda}\,\theta_i\overline\varphi_\lambda, \hspace{1ex}\mbox{so}\hspace{1ex}\psi_1(t)\lrcorner u = \sum\limits_{i=1}^5 (-1)^{i-1}\,\sum\limits_{\lambda=1}^2 t_{i\lambda}\,\overline\varphi_\lambda\wedge\widehat\varphi_i,$$ $$\rho_1(s)=\sum\limits_{j=1}^5\sum\limits_{\mu=1}^2 s_{j\mu}\,\theta_j\overline\varphi_\mu, \hspace{1ex}\mbox{so}\hspace{1ex}\rho_1(s)\lrcorner u = \sum\limits_{j=1}^5 (-1)^{j-1}\,\sum\limits_{\mu=1}^2 s_{j\mu}\,\overline\varphi_\mu\wedge\widehat\varphi_j,$$

\noindent where $\widehat\varphi_j$ stands for $\varphi_1\wedge\dots\wedge\varphi_5$ with $\varphi_j$ omitted.

Since $[\theta_i,\,\theta_j]=0$ unless $\{i,\,j\}\subset\{1,\,2,\,3\}$ and given the other values for $[\theta_i,\,\theta_j]$, we get: \begin{eqnarray*}[\psi_1(t),\,\rho_1(s)]\lrcorner u & = & -D_3(t,\,s)\,\varphi_1\wedge\varphi_2\wedge\varphi_4\wedge\varphi_5\wedge\overline\varphi_1\wedge\overline\varphi_2 + D_2(t,\,s)\,\varphi_1\wedge\varphi_2\wedge\varphi_3\wedge\varphi_5\wedge\overline\varphi_1\wedge\overline\varphi_2 \\
 & - & D_1(t,\,s)\,\varphi_1\wedge\varphi_2\wedge\varphi_3\wedge\varphi_4\wedge\overline\varphi_1\wedge\overline\varphi_2,  \hspace{3ex} \mbox{where}\end{eqnarray*}

\noindent $D_3(t,\,s) = \begin{vmatrix}
                              t_{11} & t_{12} \\
                              s_{21} & s_{22}    
                              \end{vmatrix} - \begin{vmatrix}
                                               s_{11} & s_{12} \\
                                               t_{21} & t_{22}    
                                                \end{vmatrix}$, $D_2(t,\,s) = \begin{vmatrix}
                              t_{11} & t_{12} \\
                              s_{31} & s_{32}    
                              \end{vmatrix} - \begin{vmatrix}
                                               s_{11} & s_{12} \\
                                               t_{31} & t_{32}    
                                                \end{vmatrix}$, $D_1(t,\,s) = \begin{vmatrix}

                              t_{21} & t_{22} \\

                              s_{31} & s_{32}

                              \end{vmatrix} - \begin{vmatrix}

                                               s_{21} & s_{22} \\

                                               t_{31} & t_{32}

                                                \end{vmatrix}$.

\vspace{2ex}

\noindent Now, since $\varphi_1\wedge\varphi_2 = \partial\varphi_3$ and $\bar\varphi_1\wedge\overline\varphi_2=\bar\partial\overline\varphi_3$, using also the other properties of the $\varphi_i$'s, we get \begin{eqnarray*}\varphi_1\wedge\varphi_2\wedge\varphi_4\wedge\varphi_5\wedge\overline\varphi_1\wedge\overline\varphi_2  & = & \bar\partial\partial(\varphi_3\wedge\varphi_4\wedge\varphi_5\wedge\overline\varphi_3)\\
\varphi_1\wedge\varphi_2\wedge\varphi_3\wedge\varphi_5\wedge\overline\varphi_1\wedge\overline\varphi_2  & = & \bar\partial\partial(\varphi_2\wedge\varphi_4\wedge\varphi_5\wedge\overline\varphi_3).\end{eqnarray*}

\noindent Similarly, since $\varphi_2\wedge\varphi_3 = \partial\varphi_5$ and $\bar\varphi_1\wedge\overline\varphi_2=\bar\partial\overline\varphi_3$, we get \begin{eqnarray*}\varphi_1\wedge\varphi_2\wedge\varphi_3\wedge\varphi_4\wedge\overline\varphi_1\wedge\overline\varphi_2 = \bar\partial\partial(\varphi_1\wedge\varphi_4\wedge\varphi_5\wedge\overline\varphi_3).\end{eqnarray*}

\noindent We conclude that $\partial(\psi_1(t)\lrcorner(\rho_1(s)\lrcorner u)) =[\psi_1(t),\,\rho_1(s)]\lrcorner u \in\mbox{Im}\,(\partial\bar\partial)\subset\mbox{Im}\,\bar\partial$. Meanwhile, $\psi_1(t)\lrcorner(\rho_1(s)\lrcorner u)$ is $\bar\partial$-closed (because $\psi_1(t)\lrcorner u$ and $\rho_1(s)\lrcorner u$ are), hence $\psi_1(t)\lrcorner(\rho_1(s)\lrcorner u)\in{\cal Z}_2^{4,\,1}$, as desired.

\vspace{2ex}

(b)\, Let $\psi_1(t), \rho_1(s)\in C^\infty_{0,\,1}(X,\,T^{1,\,0}X)$ such that $\psi_1(t)\lrcorner u\in\ker d$ and $\rho_1(s)\lrcorner u\in\mbox{Im}\,\partial$. Then, $$\psi_1(t)=\sum\limits_{i=1}^5\sum\limits_{\lambda=1}^2 t_{i\lambda}\,\theta_i\overline\varphi_\lambda, \hspace{1ex}\mbox{so}\hspace{1ex}\psi_1(t)\lrcorner u = \sum\limits_{i=1}^5 (-1)^{i-1}\,\sum\limits_{\lambda=1}^2 t_{i\lambda}\,\overline\varphi_\lambda\wedge\widehat\varphi_i,$$ $$\rho_1(s)=\sum\limits_{j=3}^5 s_j\,\theta_j\overline\varphi_3, \hspace{1ex}\mbox{so}\hspace{1ex}\rho_1(s)\lrcorner u = \sum\limits_{j=3}^5 (-1)^{j-1}\, s_j\,\overline\varphi_3\wedge\widehat\varphi_j.$$

\noindent  Indeed, in the case of $\rho_1(s)\lrcorner u$, we have

\vspace{1ex}

\hspace{9ex} $\widehat\varphi_3 = \partial(\varphi_3\wedge\varphi_4\wedge\varphi_5)$, so $\overline\varphi_3\wedge\widehat\varphi_3 = -\partial(\overline\varphi_3\wedge\varphi_3\wedge\varphi_4\wedge\varphi_5)$,

\vspace{1ex}

\hspace{9ex} $\widehat\varphi_4 = \partial(\varphi_2\wedge\varphi_4\wedge\varphi_5)$, so $\overline\varphi_3\wedge\widehat\varphi_4 = -\partial(\overline\varphi_3\wedge\varphi_2\wedge\varphi_4\wedge\varphi_5)$,

\vspace{1ex}

\hspace{9ex} $\widehat\varphi_5 = \partial(\varphi_1\wedge\varphi_4\wedge\varphi_5)$, so $\overline\varphi_3\wedge\widehat\varphi_5 = -\partial(\overline\varphi_3\wedge\varphi_1\wedge\varphi_4\wedge\varphi_5)$

\vspace{1ex}

\noindent and every $\partial$-exact $(4,\,1)$-form is a linear combination of $\overline\varphi_3\wedge\widehat\varphi_3$, $\overline\varphi_3\wedge\widehat\varphi_4$ and $\overline\varphi_3\wedge\widehat\varphi_5$.

On the one hand, we get $$\psi_1(t)\lrcorner(\rho_1(s)\lrcorner u) = \sum\limits_{i=1}^5\sum\limits_{j=3}^5\sum\limits_{\lambda=1}^2 (-1)^{j-1}\, t_{i\lambda}\,s_j\,\overline\varphi_\lambda\wedge\overline\varphi_3\wedge(\theta_i\lrcorner\widehat\varphi_j).$$

\noindent Now, $\theta_i\lrcorner\widehat\varphi_j$ is always $\bar\partial$-closed because it vanishes when $i=j$, it equals $(-1)^{i-1}\widehat\varphi_{ij}$ when $i<j$ and it equals $(-1)^i\widehat\varphi_{ji}$ when $i>j$, where $\widehat\varphi_{ij}$ stands for $\varphi_1\wedge\dots\wedge\varphi_5$ with $\varphi_i$ and $\varphi_j$ omitted and $i<j$. All the $\varphi_i$'s being $\bar\partial$-closed, so are all the $\widehat\varphi_{ij}$'s. Meanwhile, $\bar\partial(\overline\varphi_\lambda\wedge\overline\varphi_3)=-\overline\varphi_\lambda\wedge\bar\partial\overline\varphi_3 = 0$ for all $\lambda\in\{1,\,2\}$, since $\bar\partial\overline\varphi_3 = \overline\varphi_1\wedge\overline\varphi_2$. This proves that $\psi_1(t)\lrcorner(\rho_1(s)\lrcorner u)\in\ker\bar\partial$.

On the other hand, we get \begin{eqnarray*}\partial(\psi_1(t)\lrcorner(\rho_1(s)\lrcorner u)) & = & [\psi_1(t),\,\rho_1(s)]\lrcorner u = \sum\limits_{i=1}^5\sum\limits_{j=3}^5\sum\limits_{\lambda=1}^2 t_{i\lambda}\,s_j\,\overline\varphi_\lambda\wedge\overline\varphi_3\wedge([\theta_i,\,\theta_j]\lrcorner u) \\
  & = & - \sum\limits_{\lambda=1}^2 t_{1\lambda}\,s_3\,\overline\varphi_\lambda\wedge\overline\varphi_3\wedge(\theta_4\lrcorner u) - \sum\limits_{\lambda=1}^2 t_{2\lambda}\,s_3\,\overline\varphi_\lambda\wedge\overline\varphi_3\wedge(\theta_5\lrcorner u)  \\
  & = & t_{11}\,s_3\,\overline\varphi_1\wedge\overline\varphi_3\wedge\widehat\varphi_4 +  t_{12}\,s_3\,\overline\varphi_2\wedge\overline\varphi_3\wedge\widehat\varphi_4 - t_{21}\,s_3\,\overline\varphi_1\wedge\overline\varphi_3\wedge\widehat\varphi_5 - t_{22}\,s_3\,\overline\varphi_2\wedge\overline\varphi_3\wedge\widehat\varphi_5  \\
  & = & t_{11}\,s_3\,\bar\partial\overline\varphi_4\wedge\widehat\varphi_4 +  t_{12}\,s_3\,\bar\partial\overline\varphi_5\wedge\widehat\varphi_4 - t_{21}\,s_3\,\bar\partial\overline\varphi_4\wedge\widehat\varphi_5 - t_{22}\,s_3\,\bar\partial\overline\varphi_5\wedge\widehat\varphi_5  \\
 & = & \bar\partial(t_{11}\,s_3\,\overline\varphi_4\wedge\widehat\varphi_4 +  t_{12}\,s_3\,\overline\varphi_5\wedge\widehat\varphi_4 - t_{21}\,s_3\,\overline\varphi_4\wedge\widehat\varphi_5 - t_{22}\,s_3\,\overline\varphi_5\wedge\widehat\varphi_5)\in\mbox{Im}\,\bar\partial,\end{eqnarray*}

\noindent where the second line followed from the fact that $[\theta_i,\,\theta_j]=0$ unless $i,j\in\{1,\,2,\,3\}$ and $i\neq j$. Given the fact that the summation bears over $j\in\{3,\,4,\, 5\}$, this forces $j=3$ and $i\in\{1,\,2\}$. Then, we get the second line from $[\theta_1,\,\theta_3] = -\theta_4$ and $[\theta_2,\,\theta_3] = -\theta_5$.

The facts that $\psi_1(t)\lrcorner(\rho_1(s)\lrcorner u)\in\ker\bar\partial$ and $\partial(\psi_1(t)\lrcorner(\rho_1(s)\lrcorner u))\in\mbox{Im}\,\bar\partial$ translate to $\psi_1(t)\lrcorner(\rho_1(s)\lrcorner u)\in{\cal Z}_2^{3,\,2}$, as desired.

\vspace{2ex}

(c)\, Let $\psi_1(t), \rho_1(s)\in C^\infty_{0,\,1}(X,\,T^{1,\,0}X)$ such that $\psi_1(t)\lrcorner u, \rho_1(s)\lrcorner u\in\mbox{Im}\,\partial$. Then, $$\psi_1(t)=\sum\limits_{i=3}^5 t_i\,\theta_i\overline\varphi_3, \hspace{2ex} \rho_1(s)=\sum\limits_{j=3}^5 s_j\,\theta_j\overline\varphi_3, \hspace{1ex}\mbox{so}\hspace{1ex}\rho_1(s)\lrcorner u = \sum\limits_{j=3}^5 (-1)^{j-1}\, s_j\,\overline\varphi_3\wedge\widehat\varphi_j.$$

\noindent We get $$\psi_1(t)\lrcorner(\rho_1(s)\lrcorner u) = \sum\limits_{i=3}^5\sum\limits_{j=3}^5 (-1)^{j-1}\, t_i\, s_j\,\overline\varphi_3\wedge\overline\varphi_3\wedge(\theta_i\lrcorner\widehat\varphi_j)=0\in{\cal Z}_2^{3,\,2},$$

\noindent as desired. This completes the proof of Lemma \ref{Lem:Iwasawas_def-hypotheses}.\hfill $\Box$

\vspace{3ex}

\noindent {\bf Acknowledgements.} This work has been partially supported by the projects MTM2017-85649-P (AEI/FEDER, UE), and E22-17R ``\'Algebra y Geometr\'{\i}a'' (Gobierno de Arag\'on/FEDER)."

\vspace{2ex}

\noindent {\bf References.} \\

\noindent [AK17]\, D. Angella, H. Kasuya --- {\it Bott-Chern Cohomology of Solvmanifolds} --- Ann. Glob. Anal. Geom. {\bf 52}, no. 4 (2017), 363--411.

\vspace{1ex}

\noindent [Has10]\, K. Hasegawa --- {\it Small Deformations and Non-left-invariant Complex Structures on Six-dimensional Compact Solvmanifolds} --- Diff. Geom. App. {\bf 28}, no. 2 (2010), 220-227.

\vspace{1ex}

\noindent [Kaw92]\, Y. Kawamata --- {\it Unobstructed Deformations -- a Remark on a Paper of Z. Ran.} ---  J. Alg.Geom {\bf 1} (1992), 183-190. Erratum in J. Alg.Geom {\bf 6} (1997), 803-804.

\vspace{1ex}

\noindent [KS60]\, K. Kodaira, D.C. Spencer --- {\it On Deformations of Complex Analytic Structures, III. Stability Theorems for Complex Structures} --- Ann. of Math. {\bf 71}, no.1 (1960), 43-76.

\vspace{1ex}

\noindent [Kur62]\, M. Kuranishi --- {\it On the Locally Complete Families of Complex Analytic Structures} --- Ann. of Math. {\bf 75}, no. 3 (1962), 536-577.

\vspace{1ex}

\noindent [Nak75]\, I. Nakamura --- {\it Complex parallelisable manifolds and their small deformations} --- Journal of Differential Geometry {\bf 10} (1975), 85-112.

\vspace{1ex}

\noindent [Pop13]\, D. Popovici --- {\it Holomorphic Deformations of Balanced Calabi-Yau $\partial\bar\partial$-Manifolds} --- Ann. Inst. Fourier, 69 (2019) no. 2, pp. 673-728. doi : 10.5802/aif.3254.

\vspace{1ex}

\noindent [Pop15]\, D. Popovici --- {\it Aeppli Cohomology Classes Associated with Gauduchon Metrics on Compact Complex Manifolds} --- Bull. Soc. Math. France {\bf 143} (3), (2015), p. 1-37.

\vspace{1ex}

\noindent [Pop16]\, D. Popovici --- {\it Degeneration at $E_2$ of Certain Spectral Sequences} ---  International Journal of Mathematics {\bf 27}, no. 13 (2016), doi: 10.1142/S0129167X16501111.

\vspace{1ex}

\noindent [Pop18]\, D. Popovici --- {\it Non-K\"ahler Mirror Symmetry of the Iwasawa Manifold} --- International Mathematics Research Notices (IMRN), doi: 10.1093/imrn/rny256.

\vspace{1ex}

\noindent [PSU20a]\, D. Popovici, J. Stelzig, L. Ugarte --- {\it Higher-Page Hodge Theory of Compact Complex Manifolds} --- arXiv:2001.02313v2.

\vspace{1ex}

\noindent [PSU20b]\, D. Popovici, J. Stelzig, L. Ugarte --- {\it Higher-Page Bott-Chern and Aeppli Cohomologies and Applications} --- J. reine angew. Math., doi: 10.1515/crelle-2021-0014.

\vspace{1ex}

\noindent [Ran92]\, Z. Ran --- {\it Deformations of Manifolds with Torsion or Negative Canonical Bundle} --- J. Alg. Geom. {\bf 1} (1992), no. 2, 279-291.

\vspace{1ex}

\noindent [Rol11]\, S. Rollenske --- {\it The Kuranishi Space of Complex Parallelisable Nilmanifolds} --- J. Eur. Math. Soc. {\bf 13} (2011), 513--531.

\vspace{1ex}

\noindent [Sak76]\, Y. Sakane --- {\it On Compact Complex Parallelisable Solvmanifolds} --- Osaka J. Math. {\bf 13} (1976), 187--212.

\vspace{1ex}

\noindent [Sch07]\, M. Schweitzer --- {\it Autour de la cohomologie de Bott-Chern} --- arXiv:0709.3528v1.

\vspace{1ex}

\noindent [Ste21]\, J. Stelzig --- {\it On the Structure of Double Complexes} --- J. London Math. Soc. \\doi: 10.1112/jlms.12453.

\vspace{1ex}

\noindent [Tia87]\, G. Tian --- {\it Smoothness of the Universal Deformation Space of Compact Calabi-Yau Manifolds and Its Petersson-Weil Metric} --- Mathematical Aspects of String Theory (San Diego, 1986), Adv. Ser. Math. Phys. 1, World Sci. Publishing, Singapore (1987), 629--646.

\vspace{1ex}

\noindent [Tod89]\, A. N. Todorov --- {\it The Weil-Petersson Geometry of the Moduli Space of $SU(n\geq 3)$ (Calabi-Yau) Manifolds I} --- Comm. Math. Phys. {\bf 126} (1989), 325-346.

\vspace{1ex}

\noindent [Wan54]\, H.-C. Wang --- {\it Complex Parallisable Manifolds} --- 

Proc. Amer. Math. Soc. {\bf 5} (1954), 771--776.

\vspace{2ex}

\noindent Institut de Math\'ematiques de Toulouse,       \hfill Mathematisches Institut

\noindent Universit\'e Paul Sabatier,                    \hfill  Ludwig-Maximilians-Universit\"at

\noindent 118 route de Narbonne, 31062 Toulouse, France  \hfill Theresienstr. 39, 80333 M\"unchen, Germany

\noindent Email: popovici@math.univ-toulouse.fr          \hfill Email: Jonas.Stelzig@math.lmu.de

\vspace{1ex}

\noindent and

\vspace{1ex}

\noindent Departamento de Matem\'aticas\,-\,I.U.M.A.,

\noindent Universidad de Zaragoza,

\noindent Campus Plaza San Francisco, 50009 Zaragoza, Spain

\noindent Email: ugarte@unizar.es

\end{document}